\documentclass[10pt,a4paper]{article}

\usepackage[a4paper,hmargin=35mm,vmargin=40mm,footskip=15mm]{geometry}
\usepackage[OT1]{fontenc}
\usepackage[utf8]{inputenc}
\usepackage[english]{babel}
\usepackage{csquotes}
\usepackage{lmodern}
\usepackage{amsmath}
\usepackage{amsfonts}
\usepackage{amssymb}
\usepackage{mathtools}
\usepackage{amsthm}
\usepackage{thmtools}
\usepackage{bbm}
\allowdisplaybreaks
\usepackage{float}
\usepackage[font={small,footnotesize}]{caption}
\usepackage{subcaption}
\usepackage{booktabs}
\usepackage{tabularx}
\usepackage{tikz}
\usepackage{pgfplots}
\usepackage{color}
\pgfplotsset{compat=1.14}
\usetikzlibrary{arrows.meta}
\usetikzlibrary{patterns}
\usepackage{enumitem}
\usepackage[backend=biber,style=numeric,sorting=nyt,sortcites=true,doi=false,isbn=false,url=false,eprint=true,maxbibnames=9,giveninits=true]{biblatex}
\DeclareNameAlias{sortname}{last-first}
\renewbibmacro{in:}{}
\AtEveryBibitem{%
  \ifentrytype{misc}{%
  }{%
    \clearfield{archivePrefix}%
    \clearfield{arxivId}%
    \clearfield{eprint}%
    \clearfield{primaryClass}%
  }%
}
\usepackage{hyperref}
\hypersetup{
colorlinks=false,
bookmarksnumbered=false,
pdfpagemode=UseNone,
pdfstartpage=1,
pdfstartview={XYZ null null 1.00},
pdfpagelayout=OneColumn,
pdflang=English}
\usepackage{authblk}

\theoremstyle{plain}
\newtheorem{theorem}{Theorem}

\newtheorem{lemma}{Lemma}

\theoremstyle{definition}
\newtheorem{definition}{Definition}
\newtheorem{assumption}{Assumption}

\newtheorem{examplex}{Example}
{\pushQED{\qed}\renewcommand{\qedsymbol}{$\triangle$}\examplex}
{\popQED\endexamplex\vspace{2ex}}

\setlist{align=right,labelindent=0.2em,labelwidth=1.5em,labelsep*=0.6em,leftmargin=!,topsep=0ex,partopsep=0ex,parsep=0ex,itemsep=0ex}
\setlist[itemize,1]{label=\raisebox{0.0em}{\scalebox{1.0}{$\bullet$}}}
\setlist[itemize,2]{label=\raisebox{0.1em}{\scalebox{0.5}{$\blacksquare$}}}
\setlist[itemize,3]{label=\raisebox{0.1em}{\scalebox{0.7}{$\blacktriangleright$}}}
\setlist[itemize,4]{label=\raisebox{0.1em}{\scalebox{0.6}{$\blacklozenge$}}}
\setlist[enumerate]{label=(\alph*)}

\renewcommand{\P}{\ensuremath{\mathbb{P}}}
\newcommand{\E}{\ensuremath{\mathbb{E}}}
\newcommand{\R}{\ensuremath{\mathbb{R}}}

\newcommand{\bigO}{\ensuremath{\mathcal{O}}}

\newcommand{\smallO}{\ensuremath{o}}

\newcommand{\1}[1]{\ensuremath{\mathbbm{1}{\raisebox{-0.2ex}{\hspace{-0.1em}\scriptsize\{$#1$\}}}}}

\DeclareMathOperator{\lambertW}{\scalebox{0.85}{$\mathcal{W}$}}

\newcommand{\w}{W}
\newcommand{\wt}{\widetilde{\w}}
\renewcommand{\sf}[1][n]{\lambda_{#1}}
\renewcommand{\rm}[2][n]{c_{{#1},{#2}}}
\newcommand{\cn}{\omega}
\newcommand{\tcn}[1][n]{\cn_{#1}}
\newcommand{\te}[1][n]{\mathcal{T}_{#1,\delta}}
\newcommand{\teinc}[1][n]{\mathcal{T}_{#1,\eta}}
\newcommand{\rgm}{\mathbb{G}({n}; {\w}, {\sf})}
\newcommand{\hb}{t}

\renewcommand{\epsilon}{\varepsilon}
\renewcommand{\phi}{\varphi}

\title{Cliques in rank-1 random graphs:\\the role of inhomogeneity}
\author[{}\hspace{0.5pt}\protect\hyperlink{hyp:affil1}{a},\protect\hyperlink{hyp:email1}{1}]{\protect\hypertarget{hyp:author1}{Kay Bogerd}}
\author[{}\hspace{0.5pt}\protect\hyperlink{hyp:affil1}{a},\protect\hyperlink{hyp:email2}{2}]{\protect\hypertarget{hyp:author2}{Rui M. Castro}}
\author[{}\hspace{0.5pt}\protect\hyperlink{hyp:affil1}{a},\protect\hyperlink{hyp:email3}{3}]{\protect\hypertarget{hyp:author3}{Remco van der Hofstad}}
\affil[ ]{\small\textsuperscript{\protect\hypertarget{hyp:affil1}{a}}\hspace{0.5pt}Eindhoven University of Technology}
\affil[ ]{\small{}
\textsuperscript{\protect\hypertarget{hyp:email1}{1}}\hspace{0.5pt}\texttt{\footnotesize\href{mailto:k.m.bogerd@tue.nl}{k.m.bogerd@tue.nl}},
\textsuperscript{\protect\hypertarget{hyp:email2}{2}}\hspace{0.5pt}\texttt{\footnotesize\href{mailto:rmcastro@tue.nl}{rmcastro@tue.nl}},
\textsuperscript{\protect\hypertarget{hyp:email3}{3}}\hspace{0.5pt}\texttt{\footnotesize\href{mailto:r.w.v.d.hofstad@tue.nl}{r.w.v.d.hofstad@tue.nl}}}
\date{\today}

\begin{document}
\maketitle

\begin{abstract}
We study the asymptotic behavior of the clique number in rank-1 inhomogeneous random graphs, where edge probabilities between vertices are roughly proportional to the product of their vertex weights. We show that the clique number is concentrated on at most two consecutive integers, for which we provide an expression. Interestingly, the order of the clique number is primarily determined by the overall edge density, with the inhomogeneity only affecting multiplicative constants or adding at most a $\log\log(n)$ multiplicative factor. For sparse enough graphs the clique number is always bounded and the effect of inhomogeneity completely vanishes.

\vspace{0.5\baselineskip}

\end{abstract}

\section{Introduction}
\label{sec:introduction}
The clique number of a graph $G$ is the size of the largest clique (i.e.\ the largest complete subgraph) in $G$.  In an Erd\H{o}s-R\'{e}nyi random graph, edges between vertices are present with the same probability independently of one another. This is sometimes also called the \emph{homogeneous} setting because all edges have the same probability of being included. In this setting, it is well known that the clique number is highly concentrated when the graph has a large number of vertices, meaning that with high probability the clique number takes values in a small interval \cite{Matula1972,Matula1976,Grimmett1975}. In fact, Matula \cite{Matula1972} shows that the clique number converges to one of two consecutive integers, and provides an explicit formula for the asymptotic clique size.

In this work, we are interested in understanding the behavior of the clique number in \emph{inhomogeneous} random graphs, where edges have different occupation probabilities. In such random graphs, the properties of different vertices (e.g.\ their expected degree) can be radically different and can take a wide range of values. This is in contrast, for instance, with Erd\H{o}s-R\'{e}nyi random graphs, where degrees can only take values in a relative narrow range.

Our work is in part motivated by the statistical problem of community detection.  Formally, this amounts to testing whether a given graph was obtained by ``planting'' a clique, or dense subgraph, inside a random graph. Arias-Castro and Verzelen \cite{Arias-Castro2014,Arias-Castro2013a} have recently considered this problem with an Erd\H{o}s-R\'{e}nyi random graph as the underlying model. To extend these results to the inhomogeneous setting, one needs a better understanding of cliques in the corresponding null model; thus, studying the clique number in inhomogeneous random graphs is a natural starting point.

\paragraph{Related work.}
Inhomogeneous random graphs have received much attention over the past decade because they more accurately model the network structure observed in many real-world networks. The literature on this subject can be divided into sparse and dense graphs.

In the sparse setting, the edge probabilities decrease with the graph size such that the resulting graph has bounded average degree. This setting was first studied in substantial detail by Bollob\'{a}s, Janson, and Riordan \cite{Bollobas2007} in which the critical value for the existence of the giant component was established, as well as several related fundamental properties of such graphs were derived. This has sparked great interest in this model, see also \cite{VanderHofstad2017,VanderHofstad2018} and the references therein for an overview of recent results.

The dense setting (when the average degree is unbounded) leads to the theory of graphons developed by Lov\'{a}sz and Szegedy \cite{Lovasz2006}. Recently, first order results for the clique number were also obtained for this case by Dole\v{z}al, Hladk\'{y}, and M\'{a}th\'{e} \cite{Dolezal2017}, and further studied by McKinley \cite{McKinley2019}.

Inhomogeneous random graphs with an intermediate density have received less attention, although recently results about connectivity have been obtained by Devroye and Fraiman \cite{Devroye2014}, and the diameter was considered by Fraiman and Mitsche \cite{Fraiman2015}.

A special class of the inhomogeneous random graphs above are the so-called rank-1 random graphs. Here each vertex receives a weight and, conditionally on these weights, edges are present independently with probability equal to the product of their vertex weights. Many well-known random graphs fit this model, such as the Erd\H{o}s-R\'{e}nyi random graph by giving each vertex the same weight or scale-free graphs such as the Chung-Lu, Norros-Reitu, and Generalized random graphs by taking the weights from a power-law \cite{Chung2002, Chung2003, Chung2006, Norros2002, Britton2006}.

\paragraph{Our contribution.}
In this paper we show that the clique number of rank-1 inhomogeneous random graphs is concentrated on at most two consecutive integers, provided that all vertex weights are bounded away from $1$. We provide a single expression for the order of the clique number that is valid for every edge density, bringing together results of both the sparse and dense regimes.

To derive our results we essentially make use of the same methodology as Matula \cite{Matula1972}, namely using the first and second moment methods to obtain, respectively, upper and lower bounds for the clique number. The main contribution here lies in the definition of what we call the typical clique number $\tcn$, which is the point where the clique number concentrates around. This quantity is defined implicitly, and we show that this is indeed a sound definition. Furthermore, the inhomogeneity of these graphs substantially complicates the derivation of the lower bounds, which now requires significantly more effort than for Erd\H{o}s-R\'{e}nyi random graphs.

We find quite different asymptotic behaviors of the clique number depending on the edge density of the graph, although our results are more interesting when the average degree diverges. In sparse graphs, the clique number is always bounded regardless of the ``amount of inhomogeneity'', and the only parameter that affects the asymptotic clique number is the edge density. In dense graphs, the clique number behaves similarly as in an Erd\H{o}s-R\'{e}nyi random graph. Specifically, the clique number is always of order $\log(n)$, with the inhomogeneity only affecting the constants. Interestingly, graphs with intermediate edge density can be rather different, with the inhomogeneity sometimes adding a $\log\log(n)$ multiplicative factor to the clique number.

\section{Main results}
\label{sec:main_results}
In this paper, we consider a random graph model denoted by $\rgm$. This model has three parameters: the number of vertices $n$, the \emph{weight distribution} $\w$, and the \emph{scaling} $\sf$. An element of $\rgm$ is a simple graph $G = (V, E)$ that has $n \in \mathbb{N}$ vertices with vertex set $V = [n] \coloneqq \{1, \ldots, n\}$, and a random edge set $E$. Each vertex $i \in [n]$ is assigned a \emph{weight}, which is an independent copy $\w_i$ of the non-negative random variable $\w \in [0, \infty)$. In other words, $\w_i$ are i.i.d.\ non-negative random variables with the same distribution as $\w$.  Conditionally on these weights, the presence of an edge between two vertices $i,j \in [n]$, with $i \neq j$, is modeled by independent Bernoulli random variables with success probability
\begin{equation}
\label{eq:connection_probability}
p_{i,j} \coloneqq \P\left((i,j) \in E \,\middle|\, (\w_k)_{k \in [n]}\right) = \frac{\w_i}{\sf} \cdot \frac{\w_j}{\sf} \wedge 1 \,,
\end{equation}
where the scaling $\sf : \mathbb{N} \mapsto \mathbb{R}$ is a deterministic sequence. Note that the weights do not depend on the graph size $n$. This is why the introduction of the scaling $\sf$ is useful, as it allows us to naturally control the edge density of the graph. We assume that the scaling is at most of order $\sqrt{n}$, because otherwise we are in the trivial case where the graph is asymptotically almost or completely empty.

Random graph models like the classical Erd\H{o}s-R\'{e}nyi random graphs are \emph{homogeneous} in the sense that for a typical realization the degrees of all vertices tend to take a narrow range of values. Furthermore, all parts of the graph look more or less the same. However, graphs arising in real-world settings do not generally satisfy this property and tend to be \emph{inhomogeneous}, with a relatively wide range of different vertex degrees across the entire graph. In our model the weight distribution $\w$ determines the inhomogeneity of the graph, and the heavier the tails of this distribution the more inhomogeneous the graph is. Recall that the weight distribution is not a function of the graph size and without any scaling factor the resulting graphs are dense (i.e.\ with a number of edges that is quadratic in the graph size $n$). The parameter $\sf$ allows us, therefore, to control the edge density. When $\sf$ is constant, we are in the dense regime. On the other hand, when $\sf \approx \sqrt{n}$, we are in the sparse regime with a number of edges that is linear in $n$, which corresponds to graphs with finite average degree. Choices of $\sf$ in between those extremes lead to graphs of intermediate density with a number of edges more than linear but less than quadratic in $n$.

We are interested in an asymptotic characterization of these graphs as the number of vertices $n$ increases. In this paper, when limits are unspecified, they are taken as the number of vertices $n$ tends to $\infty$. We use standard asymptotic notation. For deterministic sequences $a_n$ and $b_n$, we write $a_n = \bigO(b_n)$ when $a_n / b_n$ is bounded, and $a_n = \smallO(b_n)$ when $a_n / b_n \to 0$. We say that a sequence of events holds with high probability if it holds with probability tending to $1$.

\subsection{The clique number}
\label{subsec:the_clique_number}
Our main contribution is to show that the clique number of a graph $G \sim \rgm$, denoted by $\cn(G)$, is concentrated on at most two consecutive integers provided that the following assumptions hold:
\begin{assumption}
\label{ass:weights_probability_assumption}
There exists $\delta > 0$ such that
\begin{equation}
\P\left(\max_{i \in [n]} \w_i \leq \frac{\sf}{1 + \delta}\right) \to 1 \,,
\qquad \text{as } n \to \infty \,.
\end{equation}
\end{assumption}
This assumption, which seems relatively benign, ensures that all edge probabilities are bounded away from $1$ with high probability. Alternatively, it can be regarded as a restriction on the denseness of the graph, requiring that $\sf$ grows fast enough, which causes the resulting graphs not to become too dense. Our second assumption strengthens the above for large $\sf$.

\begin{assumption}
\label{ass:weights_probability_assumption_increasing_delta}
If $\liminf_{n \to \infty} \log(\sf) / \log(n) > 0$, then for every $\eta > 0$
\begin{equation}
\P\left(\max_{i \in [n]} \w_i \leq \frac{\sf}{1 + \eta}\right) \to 1 \,,
\qquad \text{as } n \to \infty \,.
\end{equation}
\end{assumption}
Note that this assumption is only a restriction when the scaling is large (i.e. when $\sf$ is a positive power of $n$). Moreover, in many cases, Assumption~\ref{ass:weights_probability_assumption_increasing_delta} is a direct consequence of Assumption~\ref{ass:weights_probability_assumption} (e.g.\ when all or enough moments of $\w$ are finite, or when $\w$ has a regularly-varying distribution). We only need Assumption~\ref{ass:weights_probability_assumption_increasing_delta} to eliminate some pathological cases. This issue is discussed in more detail in Section~\ref{sec:discussion_and_overview}.

The clique number in graphs from our model depends both on the amount of inhomogeneity (captured by $\w$) and the average edge density (which is close to $(\E[W]/\sf)^2$). Under Assumptions \ref{ass:weights_probability_assumption} and \ref{ass:weights_probability_assumption_increasing_delta}, it turns out that the relation between the conditional moments of the weights fully characterizes the asymptotic clique number. To simplify notation define the \emph{truncated weight} $\wt$ as the random variable with distribution
\begin{equation}
\label{eq:truncated_weight}
\P\left(\wt \leq x\right)
  = \P\left(\w \leq x \,\middle|\, \w \leq \frac{\sf}{1 + \delta}\right)\,,
\qquad\text{for all } x \in \R \,,
\end{equation}
where $\delta > 0$ comes from Assumption~\ref{ass:weights_probability_assumption}. In other words, the distribution of $\wt$ is the conditional distribution of $\w$ given $\w \leq \sf / (1 + \delta)$. The relative truncated moments (abbreviated  to relative moments in the rest of the paper) are defined as follows:
\begin{definition}[Relative moments]
\label{def:relative_moments}
Given a weight $\w$ and scaling $\sf$, the $r$-th relative moment is defined by
\begin{equation}
\label{eq:relative_moments}
\rm{r}
  = \frac{\E\bigl[\w^r\,\big|\,\w \leq \frac{\sf}{1 + \delta}\bigr]}{\E\bigl[\w\,\big|\,\w \leq \frac{\sf}{1 + \delta}\bigr]^r}
  = \frac{\E\bigl[\wt^r\bigr]}{\E\bigl[\wt\bigr]^r}\,.
\end{equation}
\end{definition}

Note that the relative moments $\rm{r}$ depend on the graph size $n$, but only through the scaling $\sf$. To avoid notational clutter, we often omit the explicit dependence of $\rm{r}$ on $\delta$.

Building towards our result stating that the clique number is highly concentrated, we define next the \emph{typical clique number}. Unfortunately, the typical clique number depends in a cumbersome way on the relative moments $\rm{r-1}$. Therefore, it is only possible to give an implicit characterization in the general setting.
\begin{definition}[Typical clique number]
\label{def:typical_clique_number}
Let $\tcn \in [1, \infty)$ denote the solution in $r \geq 1$ of
\begin{equation}
\label{eq:typical_clique_number}
r = \frac{\log(n) - \log(r) + \log(\rm{r-1}) + 1}{\log\bigl(\sf \,/\, \E\bigl[\wt\bigr]\bigr)} + 1 \ .
\end{equation}
We call $\tcn$ the typical clique number of $\rgm$.
\end{definition}

Note that the typical clique number $\tcn$ needs not be an integer. Also, it is not immediately obvious that $\tcn$ is well defined because there could either be no solution or \eqref{eq:typical_clique_number} might have multiple solutions. However, the following lemma shows that the typical clique number $\tcn$ is well defined:
\begin{lemma}
\label{lem:typical_clique_number_existence}
Under Assumption~\ref{ass:weights_probability_assumption} the typical clique number $\tcn$ from Definition~\ref{def:typical_clique_number} exists and is unique.
\end{lemma}

The following theorem is our main result and shows that asymptotically almost all graphs generated by our model have a clique number that differs at most one from the typical clique number $\tcn$:
\begin{theorem}
\label{thm:clique_number_two_point_concentration}
Let $\epsilon > 0$ be arbitrary. Under Assumptions \ref{ass:weights_probability_assumption} and \ref{ass:weights_probability_assumption_increasing_delta} the clique number $\cn(G_n)$ of a random graph $G_n \sim \rgm$ satisfies
\begin{equation}
\lim_{n \to \infty} \P\left(\cn(G_n) \in \bigl[\lfloor \tcn - \epsilon \rfloor, \lfloor \tcn + \epsilon \rfloor\bigr]\right) = 1 \,,
\end{equation}
where $\tcn$ is the typical clique number from Definition~\ref{def:typical_clique_number}.
\end{theorem}

It is important to note that the typical clique number $\tcn$ depends on the $\delta$ from Assumption~\ref{ass:weights_probability_assumption} through the behavior of the truncated weights $\wt$. This might give the impression that the clique number $\cn(G_n)$ of a graph $G_n$ must also depends on $\delta$, which it obviously does not. However, provided $\delta$ is small enough to ensure that Assumption~\ref{ass:weights_probability_assumption} holds, the dependence of the typical clique number $\tcn$ on $\delta$ vanishes. This is further discussed in Section~\ref{sec:discussion_and_overview} below.

Theorem~\ref{thm:clique_number_two_point_concentration} shows that the clique number converges to at most one of two possible values with high probability, provided we take $\epsilon < 1/2$. This shows two-point concentration of the clique number for rank-1 inhomogeneous random graphs. To find the explicit values of these two points, we need to find an explicit solution of \eqref{eq:typical_clique_number}, which is generally difficult, see Appendix~\ref{sec:derivation_of_examples} for the details. To facilitate this, we give two alternative asymptotic characterizations of the typical clique number $\tcn$.
\begin{lemma}
\label{lem:typical_clique_number_alternative}
Under Assumption~\ref{ass:weights_probability_assumption} the typical clique number $\tcn$ is equal to the solution in $r$ of
\begin{equation}
\label{eq:typical_clique_number_alternative}
r = \log_b\left(n \, \rm{r-1}\right) - \log_b\log_b\left(n \, \rm{r-1}\right) + \log_b\left(e\right) + 1 + \smallO(1) \ ,
\end{equation}
where we abbreviate $b = \sf/\E[\wt]$.
\end{lemma}

Note that when the weight distribution is degenerate (i.e.\ has probability 1 on a single point) we obtain an Erd\H{o}s-R\'{e}nyi random graph, and $\rm{r-1} = 1$ for all $r \in \mathbb{N}$. Using Lemma~\ref{lem:typical_clique_number_alternative}, our result in Theorem~\ref{thm:clique_number_two_point_concentration} reduces to the main result in \cite{Matula1972}. On the other hand, when we consider inhomogeneous graphs, Lemma~\ref{lem:typical_clique_number_alternative} shows that we essentially have to rescale the number of vertices $n$ by the relative moments $\rm{r-1}$ to account for the inhomogeneity.

The second characterization pertains the setting where the scaling is such that it gives rise to relatively sparse graphs. In this case, many of the edge probabilities have become so small that the shape of the distribution $\w$ stops playing a role, and the typical clique number $\tcn$ converges to a constant independent of the weight distribution:
\begin{lemma}
\label{lem:typical_clique_number_bounded}
Let $\alpha \in (0, 1)$. Under Assumption~\ref{ass:weights_probability_assumption} the following are equivalent:
\begin{enumerate}[label={(\roman*)}]
\item The scaling satisfies $\sf = n^{\alpha + \smallO(1)}$.
\item The typical clique number satisfies $\tcn = 1 + 1 / \alpha + \smallO(1)$.
\end{enumerate}
\end{lemma}
This result states that when the typical clique number $\tcn$ converges to a constant, the scaling $\sf$ is essentially a power of $n$, and the converse is also true.

It is important to note that we required some assumptions to show two-point concentration of the clique number. A natural question to ask is whether these assumptions are strictly speaking necessary. Although we cannot formally make this statement, we can argue that Assumption~\ref{ass:weights_probability_assumption} cannot be significantly relaxed: consider for instance a graph $G_n \sim \rgm$ where the weights have a positive probability $\rho > 0$ of becoming larger than the scaling $\sf$, that is $\P(W \geq \sf) = \rho > 0$. Then the vertices belonging to these weights form a clique because the probability of an edge between any of these vertices equals $1$. Hence, the clique number will have approximately a binomial distribution, that is $\cn(G_n) \sim \text{Bin}(n,\rho)$, and we cannot expect the clique number $\cn(G)$ to be concentrated on any fixed length interval.

This shows that Assumption~\ref{ass:weights_probability_assumption} defines a rather sharp threshold. Below this threshold the clique number $\cn(G)$ is at most logarithmic and highly concentrated, whereas above this threshold the clique number has polynomial size and cannot be concentrated on any fixed length interval.

\section{Examples}
\label{sec:examples}
Theorem~\ref{thm:clique_number_two_point_concentration} shows that the typical clique number $\tcn$ must be very close to the clique number $\cn(G_n)$ of a graph $G_n \sim \rgm$. However, Definition~\ref{def:typical_clique_number} does not give an explicit expression for $\tcn$, but rather it gives an implicit definition as the solution of the fixed-point equation \eqref{eq:typical_clique_number}. Nevertheless, we may derive the asymptotic behavior of $\tcn$ for several interesting choices of weights $\w$ and scalings $\sf$, illustrating the different regimes one might encounter. Note that in most cases these derivations are far from trivial, see Appendix~\ref{sec:derivation_of_examples} for the details. Interestingly, in all the examples that we consider, the typical clique number $\tcn$ is primarily determined by the scaling, namely $\tcn \approx k_n \log_{\sf / \E[\w]}(n)$ where $k_n$ is typically just a constant but can be as large as $\bigO(\log\log n)$.

The maximum weight in the graph plays a crucial role in Assumption~\ref{ass:weights_probability_assumption} and it is directly related to the tail probabilities of the weight distribution by the relation
\begin{equation}
\label{eq:assumption_tail_behavior_relation}
\P\left(\max_{i \in [n]} \w_i \leq x\right)
  = \bigl(1 - \P\left(W > x\right)\bigr)^n \,.
\end{equation}
Because of this relation we find that the tail behavior of the weight distribution plays a key role in the asymptotic behavior of the clique number, and we identify three main classes of weight distributions based on this.

When the weight distribution has bounded support, the behavior of the clique number is very similar to an Erd\H{o}s-R\'{e}nyi random graph. For weights with unbounded support, the behavior of the clique number depends on how heavy the tails are. For weights with heavy-tailed distributions, the scaling must grow roughly as a power of $n$ to ensure that Assumption~\ref{ass:weights_probability_assumption} is satisfied. This restriction on the scaling makes the graph highly sparse, which causes the effect of inhomogeneity due to the weight distribution to disappear. Interestingly, when the weights have a light-tailed distribution, the behavior of the clique number strongly depends on the scaling $\sf$ with different regimes depending on how $\sf$ is chosen.

\subsection{Weights with bounded support}
\label{subsec:weights_with_bounded_support}
In this section, we consider the clique number $\cn(G_n)$ for graphs $G_n \sim \rgm$ with weight distributions $\w$ that have bounded support. The best-known example in this class is the Erd\H{o}s-R\'{e}nyi random graph. In our model, this corresponds to a degenerate weight distribution $\w$, with all the mass at $1$ (denoted by $\text{Degen}(1)$). Note that Assumption~\ref{ass:weights_probability_assumption} is trivially satisfied by taking $\sf \geq s$ for any constant $s > 1$, and the edge probability is simply $p_n=1/\sf^2$. In this case, the relative moments are $\rm{r-1} = 1$ for all $1 \leq r \leq n$, and we immediately see from Lemma~\ref{lem:typical_clique_number_alternative} that
\begin{align}
\begin{split}
\label{eq:typical_clique_number_erdos_renyi}
\tcn
  &= \log_{\sf}(n) - \log_{\sf}\log_{\sf}(n) + \log_{\sf}(e) + 1 + \smallO(1)\\[0.5ex]
  &= 2\log_{1/p_n}(n) - 2\log_{1/p_n}\log_{1/p_n}(n) + 2\log_{1/p_n}(e/2) + 1 + \smallO(1)\,.
\end{split}
\end{align}
This result was also obtained by Matula \cite{Matula1972}.

For other weight distributions, vertices with large weights are more likely to be in the largest clique than vertices with small weights. This idea can be used to show that the first order behavior of the clique number remains as in \eqref{eq:typical_clique_number_erdos_renyi} but with $p_n = (w_{\textup{max}}/\sf)^2$ and where $w_{\textup{max}}$ is the supremum of the support of $\w$. Therefore, for weights with bounded support, the first order behavior of the clique number remains unchanged when we replace the random weights $\w$ by the maximum of their support $w_{\text{max}}$. This happens because vertices with small weight have, asymptotically, a negligible probability of being part of the largest clique.

To see this, note that the edge probabilities are bounded by $p_{i,j} \leq (w_{\text{max}} / \sf)^2$ for all $i \neq j \in [n]$. Plugging this into \eqref{eq:typical_clique_number_erdos_renyi} gives the following high probability upper bound on the clique number $\cn(G_n)$ of a graph $G_n \sim \rgm$,
\begin{equation}
\label{eq:heuristic_upper_bound_optimal_bounded}
\cn(G_n)
  \leq (1 + \smallO(1))\,\frac{\log(n)}{\log(\sf/w_{\text{max}})}\,.
\end{equation}
To obtain a matching lower bound we can use the following simple heuristic. Instead of considering the whole graph, consider the subgraph induced by the vertices with large enough weights, in particular the subgraph induced by the vertices $U_n = \{i \in [n] : W_i > \hb_n\}$ for some $\hb_n$. On this subgraph all weights are larger than $\hb_n$, and therefore we can bound the edge probability by $p_{i,j} \geq (\hb_n / \sf)^2$ for all $i \neq j \in U_n$. Since $|U_n| \approx n\,\P(W > \hb_n)$ we can use \eqref{eq:typical_clique_number_erdos_renyi} to obtain the following high probability lower bound on the clique number,
\begin{equation}
\label{eq:heuristic_lower_bound}
\cn(G_n)
  \geq (1 + \smallO(1))\,\frac{\log(|U_n|)}{\log(\sf/\hb_n)}
  = (1 + \smallO(1))\,\frac{\log(n) + \log(\P(W > \hb_n))}{\log(\sf/\hb_n)}\,.
\end{equation}
Note that this lower bound holds for every $\hb_n$, so we can find an optimal $\hb_n$ that maximizes the right-hand side of \eqref{eq:heuristic_lower_bound}. For weights with bounded support, taking $\hb_n = w_{\text{max}} - \smallO(1)$ suffices, provided that the $\smallO(1)$ term vanishes slowly enough to ensure that $\log\P(W > \hb_n) = \smallO(\log(n))$. This gives
\begin{equation}
\label{eq:heuristic_bound_optimal_bounded}
\cn(G_n)
  = (1 + \smallO(1))\,\frac{\log(n)}{\log(\sf/w_{\text{max}})}\,.
\end{equation}
This is precisely the leading order behavior in \eqref{eq:typical_clique_number_erdos_renyi}, but with $p_n = (w_{\textup{max}}/\sf)^2$.

Determining the asymptotics of the clique number $\cn(G_n)$ more accurately requires significantly more effort, because the argument above no longer suffices and one needs to actually solve \eqref{eq:typical_clique_number} from Definition~\ref{def:typical_clique_number}. In Table~\ref{tbl:typical_clique_number_short_tailed} the typical clique number is shown for some weight distributions. As explained above, the first order behavior is the same in all these examples. However, it can clearly be seen that weights with less mass around the maximum of their support have a smaller second order term, as expected.

\begin{table}[H]
\centering
\small
\caption{The asymptotic behavior of the typical clique number $\tcn$ for some weights $\w$ with support on $[0, 1]$. Here $\Gamma(\cdot)$ denotes the gamma function. See Appendix~\ref{sec:derivation_of_examples} for the derivation of these results.}
\label{tbl:typical_clique_number_short_tailed}
\begin{tabularx}{0.95\textwidth}{p{2.2cm}X}
\toprule
$\w$ & Typical clique number $\tcn$\\
\midrule
\addlinespace
$\text{Degen}(1)$ & $\log_{\sf}(n) - \log_{\sf}\log_{\sf}(n) + \log_{\sf}(e) + 1 + \smallO(1)$\\
\addlinespace
$\text{Ber}(p)$ & $\log_{\sf}(n\,p) - \log_{\sf}\log_{\sf}(n\,p) + \log_{\sf}(e) + 1 + \smallO(1)$\\
\addlinespace
$\text{Unif}(0,1)$ & $\log_{\sf}(n) - 2\log_{\sf}\log_{\sf}(n) + \log_{\sf}(e) + 1 + \smallO(1)$\\
\addlinespace
$\text{Beta}(\alpha,\beta)$ & $\log_{\sf}(n) - (1+\beta) \log_{\sf}\bigl( (1 + \beta) \log_{\sf}(n)\bigr)$\\[0.6ex]
 & \multicolumn{1}{r}{${} + \log_{\sf}(e) + \log_{\sf}(\Gamma(\alpha+\beta) / \Gamma(\alpha)) + 1 + \smallO(1)$}\\
\addlinespace
\bottomrule
\end{tabularx}
\end{table}

\subsection{Weights with light tails}
\label{subsec:weights_with_light_tails}
In this section, we consider the clique number $\cn(G_n)$ for graphs $G_n \sim \rgm$ with weight distributions $\w$ that have unbounded support but light tails. This is arguably the most interesting setting, and where the effect of inhomogeneity is the most pronounced. For such weight distributions the maximum weight $M_n \coloneqq \max_{i \in [n]} \w_i$ is typically highly concentrated around its expectation $\E[M_n]$. Therefore, we can choose any scaling $\sf$ slightly larger than $\E[M_n]$ to satisfy Assumption~\ref{ass:weights_probability_assumption}.
For this class of distributions we observe two very distinct behaviors depending on the choice of scaling $\sf$.

The slowest scaling that still ensures that Assumption~\ref{ass:weights_probability_assumption} is satisfied is $\sf \approx (1+\phi)\E[M_n]$, with $\phi > 0$. For a given weight distribution, this is the densest graph for which we can apply Theorem~\ref{thm:clique_number_two_point_concentration}. In this case, we see that the shape of the weight distribution has a real impact on the asymptotic behavior of the typical clique number $\tcn$, as shown in Table~\ref{tbl:typical_clique_number_light_tailed1}. In other words, the asymptotic behavior of the clique number depends on the chosen weight distribution, and this amounts to more than constant multiplicative factors in the various terms.

We consider three distributions, namely the half-normal, the Gamma and the log-normal. For both the half-normal and Gamma distribution we see that the typical clique number is of order $\log(n)$, whereas in an Erd\H{o}s-R\'{e}nyi random graph with the same edge density the cliques are smaller, of order $\log(n) / \log\log(n)$. For log-normal weights the first order behavior of the typical clique number $\tcn$ is the same in the corresponding Erd\H{o}s-R\'{e}nyi random graph although with different constants. Note that the effect of inhomogeneity is much weaker for log-normal weights. Because the log-normal distribution is ``nearly'' heavy tailed, this is consistent with our findings in the next section, where we show that, for heavy-tailed weights, the specific shape of the distribution is not relevant anymore.

If the scaling is such that Assumption~\ref{ass:weights_probability_assumption} is more easily satisfied (and therefore resulting in sparser graphs) then the contribution of the weight distribution becomes far less prominent. In particular, we consider $\sf \approx \E[M_n]^{1 + \phi}$, with $\phi > 0$. This choice of scaling leads to behavior that is qualitatively similar to that in Section~\ref{subsec:weights_with_bounded_support}. As soon as $\phi > 0$, the resulting graphs become so sparse that the shape of the weight distribution has no severe impact on the asymptotic behavior of the typical clique number $\tcn$, and only multiplicative factors are affected. This can be seen in Table~\ref{tbl:typical_clique_number_light_tailed2}.

The heuristic to obtain a high probability lower bound on the clique number, as explained in the previous section, also remains valid for light-tailed distributions. Interestingly, also in this case the lower bound seems to be tight. That is, for the $\hb_n$ that maximises the lower bound in \eqref{eq:heuristic_lower_bound}, we find exactly the same behavior, including the same constants, of the clique number as in Tables \ref{tbl:typical_clique_number_light_tailed1} and \ref{tbl:typical_clique_number_light_tailed2}. However, because the weights are no longer bounded from above, we are not aware of a simple method to obtain a matching upper bound. Nevertheless, we strongly suspect that this heuristic also gives the correct first order behavior of the clique number for other light-tailed distributions.

\begin{table}[H]
\centering
\small
\caption{The asymptotic behavior of the typical clique number $\tcn$ for some light-tailed weights $\w$ and scaling $\sf \approx (1 + \phi)\E[M_n]$ with $\phi > 0$ arbitrary.
For comparison we include the clique number of an Erd\H{o}s-R\'{e}nyi random graph with the same edge density, that is $p_n = \smash{(\E[\w]/\sf)^2}$. Here $\Gamma(\cdot)$ is the gamma function, and we write $\xi_{k}(\phi) = \smash{-k / \lambertW_{-1}\bigl(-1 / (e (1+\phi)^k)\bigr)} \in (0,1)$, where $\lambertW_{-1}(\cdot)$ is the lower branch of the Lambert-W function, see \eqref{eq:lambert_w}. See Appendix~\ref{sec:derivation_of_examples} for the derivation of these results.}
\label{tbl:typical_clique_number_light_tailed1}
\begin{tabularx}{0.95\textwidth}{p{2.2cm}p{3.4cm}X}
\toprule
$\w$ & $\sf$ & Typical clique number $\tcn$\\
\midrule
\addlinespace
$|\text{N}(0,\sigma)|$ & $(1+\phi) \sqrt{2 \sigma^2 \log(n)}$ &
$(1 + \smallO(1)) \: \xi_2(\phi) \log(n)$\\
\addlinespace
\multicolumn{2}{l}{Comparable Erd\H{o}s-R\'{e}nyi graph} & $(1 + \smallO(1)) \: 2 \log(n) / \log\log(n)$\\
\addlinespace
\midrule
\addlinespace
$\text{Gamma}(\alpha, \beta)$ & $(1 + \phi) \log(n) / \beta$ &
$(1 + \smallO(1)) \: \xi_1(\phi) \log(n)$\\
\addlinespace
\multicolumn{2}{l}{Comparable Erd\H{o}s-R\'{e}nyi graph} & $(1 + \smallO(1)) \: \log(n) / \log\log(n)$\\
\addlinespace
\midrule
\addlinespace
$\text{LN}(0,1)$ & $(1 + \phi) \exp(\sqrt{2 \log(n)})$ &
$(1 + \smallO(1)) \: \sqrt{2\log(n)}$\\
\addlinespace
\multicolumn{2}{l}{Comparable Erd\H{o}s-R\'{e}nyi graph} & $(1 + \smallO(1)) \: \sqrt{(1 / 2) \log(n)}$\\
\addlinespace
\bottomrule
\end{tabularx}
\end{table}

\begin{table}[H]
\centering
\small
\caption{The asymptotic behavior of the typical clique number $\tcn$ for some light-tailed weights $\w$ and scaling $\sf \approx \E[M_n]^{1+\phi}$ with $\phi > 0$ arbitrary.
For comparison we include the clique number of an Erd\H{o}s-R\'{e}nyi random graph with the same edge density, that is $p = \smash{(\E[\w]/\sf)^2}$. See Appendix~\ref{sec:derivation_of_examples} for the derivation of these results.}
\label{tbl:typical_clique_number_light_tailed2}
\begin{tabularx}{0.95\textwidth}{p{2.2cm}p{3.4cm}X}
\toprule
$\w$ & $\sf$ & Typical clique number $\tcn$\\
\midrule
\addlinespace
$|\text{N}(0,\sigma)|$ & $\sqrt{2 \sigma^2 \log(n)}^{1 + \phi}$ & $(1 + \smallO(1)) \: (2 / \phi) \, (\log(n) / \log\log(n))$\\
\addlinespace
\multicolumn{2}{l}{Comparable Erd\H{o}s-R\'{e}nyi graph} & $(1 + \smallO(1)) \: (2 / (1 + \phi)) \, (\log(n) / \log\log(n))$\\
\addlinespace
\midrule
\addlinespace
$\text{Gamma}(\alpha, \beta)$ & $\log(n)^{1 + \phi} / \beta$ & $(1 + \smallO(1)) \: (1 / \phi) (\log(n) / \log\log(n))$\\
\addlinespace
\multicolumn{2}{l}{Comparable Erd\H{o}s-R\'{e}nyi graph} & $(1 + \smallO(1)) \: (1 / (1 + \phi)) (\log(n) / \log\log(n))$\\
\addlinespace
\midrule
\addlinespace
$\text{LN}(0,1)$ & $\exp\bigl(\sqrt{2 \log(n)}\bigr)^{1 + \phi}$ & $(1 + \smallO(1)) \: \bigl((1+\phi) - \sqrt{\phi(2 + \phi)}\bigr) \sqrt{2\log(n)}$\\
\addlinespace
\multicolumn{2}{l}{Comparable Erd\H{o}s-R\'{e}nyi graph} & $(1 + \smallO(1)) \: (1 / (1 + \phi)) \sqrt{(1 / 2) \log(n)}$\\
\addlinespace
\bottomrule
\end{tabularx}
\end{table}

\subsection{Weights with heavy tails}
\label{subsec:weights_with_heavy_tails}
In this section we consider the clique number $\cn(G_n)$ for graphs $G_n \sim \rgm$ with weight distributions $\w$ that have heavy tails, which we define as distributions whose moments are not all finite. For these distributions, finding the clique number is surprisingly straightforward. To apply Theorem~\ref{thm:clique_number_two_point_concentration} we need a scaling $\sf$ such that Assumptions \ref{ass:weights_probability_assumption} and \ref{ass:weights_probability_assumption_increasing_delta} are satisfied, and we necessarily have $\sf \geq n^{\alpha + \smallO(1)}$ for some $\alpha>0$. This means that for heavy-tailed distributions we can always apply Lemma~\ref{lem:typical_clique_number_bounded}, which shows that the typical clique number $\tcn$ is bounded and completely determined by the scaling.

A notable special case of this was treated in \cite{Janson2010} and \cite{Bianconi2005,Bianconi2006}, where the clique number in scale-free graphs with a model similar to ours was considered. In those works the weights have a power-law distribution and the scaling is chosen as $\sf = \sqrt{n}$. The authors find that the clique number asymptotically becomes either $2$ or $3$ when the variance of the weights is finite. Using Lemma~\ref{lem:typical_clique_number_bounded} we first determine that $\tcn \to 3$, since the scaling is $\sf = \sqrt{n}$. Therefore, it follows from Theorem~\ref{thm:clique_number_two_point_concentration} that, asymptotically, the clique number must be either $2$ or $3$, precisely the same result. Note that for this scaling, having weights with finite variance and Assumptions \ref{ass:weights_probability_assumption} and \ref{ass:weights_probability_assumption_increasing_delta} are equivalent.

In the highly inhomogeneous case, where the weights have infinite variance, we require a scaling of $\sf \geq n^{\alpha + \smallO(1)}$ for some $\alpha > 1/2$ in order to satisfy Assumptions \ref{ass:weights_probability_assumption} and \ref{ass:weights_probability_assumption_increasing_delta}. However, when the scaling is this large, the resulting graphs are asymptotically almost or completely empty. On the other hand, as explained at the end of Section~\ref{sec:main_results}, when $\alpha \leq 1/2$ the clique number will approximately have a binomial distribution and thus cannot concentrate on any fixed length interval.

\section{Discussion and overview}
\label{sec:discussion_and_overview}
In this section we remark on our results and discuss some possibilities for future work.

\paragraph{Typical clique number.}
Let us first remark on our main result, Theorem~\ref{thm:clique_number_two_point_concentration}, that shows that the clique number $\cn(G_n)$ of a graph $G_n$ and the corresponding typical clique number $\tcn$ must be very close. As explained in Section~\ref{sec:main_results} the typical clique number $\tcn$ still depends on the $\delta$ from Assumption~\ref{ass:weights_probability_assumption}, whereas the clique number $\cn(G_n)$ of a graph $G_n$ obviously does not. This is certainly not desirable, and it should be possible to define the typical clique number $\tcn$ independently of $\delta$. In all examples that we considered in Section~\ref{sec:examples} this is indeed possible, since in those examples:
\begin{equation}
\label{con:relative_moments}
\frac{\E\bigl[\wt^r\bigr]}{\E\bigl[\wt\bigr]^r}
  = \frac{\E\bigl[\w^r\,\big|\,\w \leq \frac{\sf}{1 + \delta}\bigr]}{\E\bigl[\w\,\big|\,\w \leq \frac{\sf}{1 + \delta}\bigr]^r}
  = (1 + \smallO(1)) \, \frac{\E\bigl[\w^r\bigr]}{\E\bigl[\w\bigr]^r}\,,
  \qquad\text{for all } r \leq \tcn\,.
\end{equation}
We conjecture that a similar statement should hold in general, or at least for a wide class of weights $\w$ and scalings $\sf$. When proven, this would imply that the truncation in Definitions~\ref{def:relative_moments}~and~\ref{def:typical_clique_number} can be ignored. This would solve the issue of the seeming dependence between the typical clique number $\tcn$ and the $\delta$ from Assumption~\ref{ass:weights_probability_assumption}, and at the same time, make explicit computations of the typical clique number $\tcn$ somewhat easier.

\paragraph{Connection between Assumptions \ref{ass:weights_probability_assumption} and \ref{ass:weights_probability_assumption_increasing_delta}.}
Most of our results only require Assumption~\ref{ass:weights_probability_assumption}, but to prove our main result we require the slightly stronger Assumption~\ref{ass:weights_probability_assumption_increasing_delta}. However, in most cases that we checked, Assumption~\ref{ass:weights_probability_assumption_increasing_delta} is implied from Assumption~\ref{ass:weights_probability_assumption}, and the choice of weight distribution $\w$ and scaling $\sf$.

Suppose that $\sf \geq n^{\alpha + \smallO(1)}$ for some $\alpha \in (0, 1)$. Then we necessarily need to have $\E[W^{1 / \alpha}] < \infty$ in order to satisfy Assumption~\ref{ass:weights_probability_assumption}. When a slightly larger moment of $\w$ is also finite, that is $\E[W^{1 / \alpha + \epsilon}] < \infty$ for some $\epsilon > 0$, then both assumptions are simultaneously satisfied. To see this, note that by the moment condition we have $\P(W^{1 / \alpha + \epsilon} > n) \leq \smallO(1/n)$.
Now take any $\eta > 0$, then for $n$ large enough
\begin{align}
\begin{split}
\P\left(W > \frac{\sf}{1 + \eta}\right)
  &\leq \P\left(W > \frac{n^{\alpha + \smallO(1)}}{1 + \eta}\right)\\
  &\leq \P\bigl(W > n^{\alpha / (1 + \alpha \epsilon)}\bigr)
  = \P(W^{1 / \alpha + \epsilon} > n)
  \leq \smallO(1 / n) \,.
\end{split}
\end{align}
Hence, both Assumptions \ref{ass:weights_probability_assumption} and \ref{ass:weights_probability_assumption_increasing_delta} are simultaneously satisfied.

Alternatively, Assumption~\ref{ass:weights_probability_assumption} is also sufficient when $\w$ is regularly varying of index $\beta < 0$. In this case, for any $\eta > 0$, we have
\begin{equation}
\frac{\P\left(W > \frac{\sf}{1 + \eta}\right)}{\P\left(W > \frac{\sf}{1 + \delta}\right)}
  = (1 + \smallO(1)) \left(\frac{1 + \eta}{1 + \delta}\right)^{-\beta} \,.
\end{equation}
By Assumption~\ref{ass:weights_probability_assumption} we have $\smash{\P\bigl(W > \frac{\sf}{1 + \delta}\bigr)} = \smallO(1 / n)$; therefore, we also have $\smash{\P\bigl(W > \frac{\sf}{1 + \eta}\bigr)} = \smallO(1 / n)$. Hence, Assumption~\ref{ass:weights_probability_assumption_increasing_delta} is also satisfied.

\paragraph{Different models.}
In our model, the edge probabilities are $p_{i,j} = \min(X_{i,j}, 1)$, where $X_{i,j} = (\w_i / \sf) \cdot (\w_j / \sf)$. We require the minimum because otherwise some edge probabilities could exceed $1$. To achieve the same effect one has other options; some common examples are $\hat{p}_{i,j} = 1 - \exp(-X_{i,j})$ or $\tilde{p}_{i,j} = X_{i,j} / (1 + X_{i,j})$. Changing the model in this manner does not have a significant influence on the asymptotic clique number, provided Assumption~\ref{ass:weights_probability_assumption} holds. To see this, note that we can bound the edge probabilities of these models by $\min(X_{i,j}/2, 1) \leq \hat{p}_{i,j}, \tilde{p}_{i,j} \leq \min(X_{i,j}, 1)$ with high probability. Obviously, the clique number is then also bounded by the clique numbers obtained from the models with edge probabilities as given in these bounds. Since these bounds differ only by a constant multiplicative factor, it is easily seen from Definition~\ref{def:typical_clique_number} that both lead to the same leading order asymptotics of the clique number when the scaling is diverging. When the scaling is constant, the situation is more subtle and the precise clique number will change by a multiplicative factor that depends on the specific model considered.

Instead of the change in truncation described above, we could also consider different interactions between the weights. We currently only consider so-called rank-1 inhomogeneous random graphs, where the probability of an edge is proportional to the product of the weights of the incident vertices. Instead, we could model different types of interaction by considering an arbitrary symmetric function, called a kernel. It would be interesting to see whether our results can be extended to this more general setting. In particular, whether the two-point concentration of the clique number is specific to rank-1 inhomogeneous random graphs, or whether this remains true for a wider class of kernels.

When weights have bounded support, the heuristic explained in Section~\ref{subsec:weights_with_bounded_support} can be extended to obtain first order behavior of the clique number for a large class of kernels. For these kernels, this gives a simpler approach to finding the asymptotic behavior of the clique number than the method described in \cite{Dolezal2017} which provided a general answer. Moreover, based on the results in Section~\ref{subsec:weights_with_light_tails} it might also be possible to extend these results to unbounded kernels.

\paragraph{Planted clique problem.}
In the planted clique problem one starts by generating a graph as usual. After generating this graph we select a small number of vertices and connect all of them, so that they form a clique. Given such a graph with a planted clique, the problem is to locate this clique with high probability.

The work on this problem has focussed on two cases. In the first case, the underlying graph is an Erd\H{o}s-R\'{e}nyi random graph. In principle, this problem can be solved as soon as the planted clique is of size $\bigO(\log(n))$. However, if one is interested in algorithms that can recover the largest clique in polynomial time, then the best-known algorithms require the planted clique to be of size $\bigO(\sqrt{n})$, see \cite{Alon1998,Feige2000,Dekel2014}. The second case focusses on the very inhomogeneous case, with graphs that have a power-law degree distribution. Here the largest clique can be recovered in polynomial time, see \cite{Janson2010,Friedrich2015}.

Alternatively, one could consider the similar hypothesis testing problem. Here we observe a graph where it is unknown whether a clique was planted, and the problem is to decide whether it was planted or not. Instead of a clique, one could plant a denser subgraph and test whether that was planted or not, see \cite{Arias-Castro2014,Arias-Castro2013a}. Using the model from Section~\ref{sec:main_results} all these problems can be considered in a single framework. It would be particularly interesting to see what the effects of inhomogeneity and sparsity are on the computational complexity in these problems.

\mathtoolsset{showonlyrefs}
\section{Proofs}
\label{sec:proofs}
This section is devoted to proving the results in Section~\ref{sec:main_results}.  The proofs of Lemmas \ref{lem:typical_clique_number_existence}, \ref{lem:typical_clique_number_alternative}, and \ref{lem:typical_clique_number_bounded} are fairly self explanatory. To prove Theorem~\ref{thm:clique_number_two_point_concentration} we use the same approach as Matula \cite{Matula1972}, using the first and second moment method to obtain an upper and lower bound on the clique number separately.

\subsection{Proof of Lemma~\ref{lem:typical_clique_number_existence}: Existence and uniqueness of the typical clique number}
\label{subsec:existence_of_the_typical_clique_number}
Lemma~\ref{lem:typical_clique_number_existence} shows that Assumption \ref{ass:weights_probability_assumption} is sufficient to guarantee the existence and uniqueness of the typical clique number $\tcn$ in Definition~\ref{def:typical_clique_number}.  We first show that there must be at least one solution to \eqref{eq:typical_clique_number} and then show that this solution is unique.

\begin{proof}[Proof of Lemma~\ref{lem:typical_clique_number_existence}]
To simplify notation, let $f_n(r)$ be the right-hand side of \eqref{eq:typical_clique_number}, that is
\begin{equation}
\label{eq:right_hand_side_typical_clique_number}
f_n(r) = \frac{\log\left(n\right) - \log\left(r\right) + \log\left(\rm{r-1}\right) + 1}{\log\left(\sf/\E[\wt]\right)} + 1 \,.
\end{equation}
To prove the lemma we must show that the solution set of \eqref{eq:typical_clique_number}, given by $\{r\geq 1 : r=f_n(r)\}$, is non-empty and consists of a single point.  First note that $\rm{r-1}$ is a continuous function in $r$ (since these are relative moments of a \emph{truncated} distribution). This in turn implies that  $f_n(r)$ is continuous in $r$. To ensure that the solution set is non-empty, first note that
\begin{equation}
f_n(1)
  = \frac{\log\left(n\right) + 1}{\log\left(\sf/\E[\wt]\right)} + 1 \geq 1 \,,
\end{equation}
and
\begin{equation}
f_n(n)
  = \frac{\log\left(\rm{n-1}\right) + 1}{\log\left(\sf/\E[\wt]\right)} + 1
  \leq (n-1)\frac{\log\left(\frac{\sf}{1 + \delta} \middle/ \E[\wt]\right)}{\log\left(\sf \middle/ \E[\wt]\right)} + 1
  \leq n \,.
\end{equation}
Hence, there exists at least one value $r \in [1,n]$ satisfying $r = f_n(r)$. To show the uniqueness of this solution we simply show that the slope of $f_n(r)$ is strictly smaller than 1. Note that
\begin{align}
\frac{\partial}{\partial r} f_n(r)
  &= \frac{(\rm{r-1}' / \rm{r-1}) - (1/r)}{\log(\sf / \E[\wt])}\\
  &= \frac{\left(\E\bigl[\log\left(\wt \big/ \E[\wt]\right) \wt^{r-1}\bigr] \middle/ \E\bigl[\wt^{r-1}\bigr]\right) - (1 / r)}{\log(\sf / \E[\wt])}\\
\label{eq:right_hand_side_typical_clique_number_derivative}
  &\leq \frac{\log\left(\frac{\sf}{1 + \delta} \middle/ \E[\wt]\right) - (1 / r)}{\log\left(\sf \middle/ \E[\wt]\right)}
  < 1\,,
\end{align}
where $\rm{r-1}'$ denotes the partial derivative of $\rm{r-1}$ with respect to $r$. Since the partial derivative of $f_n(r)$ is strictly less than $1$, there can be at most a single solution of $r = f_n(r)$. Hence, the typical clique number is well defined.
\end{proof}

\subsection{Proof of Lemma~\ref{lem:typical_clique_number_alternative}: Alternative characterization of the typical clique number}
\label{subsec:alternative_representation_of_the_typical_clique_number}
Here we derive an alternative representation for the typical clique number $\tcn$, as formulated in Lemma~\ref{lem:typical_clique_number_alternative}. This is sometimes more convenient than the original in Definition~\ref{def:typical_clique_number}.

\begin{proof}[Proof of Lemma~\ref{lem:typical_clique_number_alternative}]
By Lemma~\ref{lem:typical_clique_number_existence} we know that the typical clique number $\tcn$ exists. Therefore, we can solve \eqref{eq:typical_clique_number} to see that the typical clique number is also the solution of
\begin{equation}
r
  = \frac{\lambertW_0\left(n \, \rm{r-1} \, e \, b \, \log(b)\right)}{\log(b)} \,,
\end{equation}
where $b = \sf/\E[\wt]$ and $\lambertW_0(\cdot)$ denotes the principal branch of the Lambert-W function, see \eqref{eq:lambert_w}. Using the approximation $\lambertW_0(x) = \log(x) - \log\log(x) + \smallO(1)$ as $x \to \infty$, as shown in \cite{Corless1996}, we obtain
\begin{align}
r &= \log_b\left(n \, \rm{r-1} \, e \, b \log(b)\right) - \log_b\log\left(n \, \rm{r-1} \, e \, b \log(b)\right) + \smallO(1)\\[0.5ex]
  &= \log_b\left(n \, \rm{r-1}\right) - \log_b\log_b\left(n \, \rm{r-1}\right) + \log_b\left(e\right) + 1 + \smallO(1) \,. \qedhere
\end{align}
\end{proof}

\subsection{Proof of Lemma~\ref{lem:typical_clique_number_bounded}: Bounded typical clique number}\label{subsec:bounded_typical_clique_number}
Here we show that the scaling $\sf$ is a positive power of $n$ if, and only if, the typical clique number $\tcn$ converges to a constant. To this end, we first derive a small lemma:

\begin{lemma}
\label{lem:typical_clique_number_bounded_additional}
Let $\alpha \in (0, 1)$. If the scaling satisfies $\sf \geq n^{\alpha + \smallO(1)}$ then
\begin{equation}
\label{eq:relative_moment_bound}
\frac{\log\bigl(\rm{1/\alpha + \smallO(1)}\bigr)}{\log\bigl(\sf / \E[\wt]\bigr)} = \smallO(1) \,.
\end{equation}
\end{lemma}
\begin{proof}
By Assumption~\ref{ass:weights_probability_assumption},
\begin{equation}
\P\left(\max_{i \in [n]} \w_i \leq \sf\right)
  = \left(1 - \P\left(\w > \sf\right)\right)^n
  \to 1 \,.
\end{equation}
Let $\epsilon > 0$ be arbitrary, then for $n$ large enough and using the above we obtain
\begin{equation}
\label{eq:weight_probability_bound}
\P\left(\w^{1 / \alpha - \epsilon} > n\right)
  \leq \P\left(\w^{1 / (\alpha + \smallO(1))} > n\right)
  = \P\left(\w > \sf\right)
  = \smallO\left(\frac{1}{n}\right) \,.
\end{equation}
Therefore, using the tail formula for expectation,
\begin{align}
\E\left[\wt^{1 / \alpha - \epsilon}\right]
  &\leq (1 + \smallO(1)) \E\left[\w^{1 / \alpha - \epsilon} \, \1{\w^{1 / \alpha + \smallO(1)} \leq n}\right]\\
  &\leq (1 + \smallO(1)) \E\left[\w^{1 / \alpha - \epsilon} \, \1{\w^{1 / \alpha - \epsilon} \leq n}\right]\\
  &= (1 + \smallO(1)) \int_{0}^{\infty} \P\left(\w^{1 / \alpha - \epsilon} \, \1{\w^{1 / \alpha - \epsilon} \leq n} > x\right) \, dx \,,\\
\intertext{where $\1{\cdot}$ denotes the usual indicator function. Note that $\w^{1 / \alpha_n} \, \1{\w^{1 / \alpha_n} \leq n} \leq n$, so we can change the upper integration limit. This gives}
\label{eq:weight_expectation_bound}
\E\left[\wt^{1 / \alpha - \epsilon}\right]
  &\leq \bigO(1) + (1 + \smallO(1)) \int_{1}^{n} \P\left(\w^{1 / \alpha - \epsilon} > x\right) \, dx\\
  &\leq \bigO(1) + (1 + \smallO(1)) \int_{1}^{n} \frac{1}{x} \, dx
  \;=\; \bigO(1) + (1 + \smallO(1)) \log\left(n\right) \,.
\end{align}
Based on the above we conclude
\begin{align}
\frac{\log\bigl(\rm{1/\alpha + \smallO(1)}\bigr)}{\log\bigl(\sf / \E[\wt]\bigr)}
  &= (1 + \smallO(1))\frac{\log\bigl(\E\bigl[\wt^{1 / \alpha + \smallO(1)}\bigr]\bigr)}{\log\bigl(\sf\bigr)}\\
  &\leq (1 + \smallO(1))\frac{\log\bigl(\E\bigl[\wt^{1 / \alpha - \epsilon}\bigr] \, \sf^{2\epsilon}\bigr)}{\log\bigl(\sf\bigr)}\\
  &= (1 + \smallO(1))\frac{\log\log\bigl(n\bigl)}{\log\bigl(n^{\alpha + \smallO(1)}\bigl)} + 2\epsilon + \smallO(1)
  = 2\epsilon + \smallO(1) \,.
\end{align}
Since $\epsilon > 0$ can be taken arbitrarily small, this completes the proof.
\end{proof}
%

We are now ready to prove Lemma~\ref{lem:typical_clique_number_bounded}. The main idea is that, as $n \to \infty$, most terms of \eqref{eq:typical_clique_number} become negligible and the remaining terms no longer involve $n$.

\begin{proof}[Proof of Lemma~\ref{lem:typical_clique_number_bounded}]
First suppose that $\sf = n^{\alpha + \smallO(1)}$ with $\alpha \in (0, 1)$. By Definition~\ref{def:typical_clique_number} the typical clique number $\tcn$ satisfies
\begin{equation}
\label{eq:typical_clique_number_simplification_part1}
\tcn
  = \frac{\log(n) - \log(\tcn) + \log(\rm{\tcn-1}) + 1}{\log\bigl(\sf/\E[\wt]\bigr)} + 1 \,.
\end{equation}
Plugging $\tcn = 1 + 1 / \alpha + \smallO(1)$ into \eqref{eq:typical_clique_number_simplification_part1} and using Lemma~\ref{lem:typical_clique_number_bounded_additional},
\begin{align}
\label{eq:typical_clique_number_simplification_part2}
\tcn &= \frac{\log(n) + \log(\rm{1/\alpha + \smallO(1)})}{\log\bigl(\sf/\E[\wt]\bigr)} + 1\\
  &= \frac{\log(n)}{\log\bigl(n^{\alpha + \smallO(1)}/\E[\wt]\bigr)} + 1 + \smallO(1)\\
  &= \frac{1}{\alpha} + 1 + \smallO(1) \,.
\end{align}
Hence \eqref{eq:typical_clique_number_simplification_part1} is satisfied for $\tcn = 1 + 1/\alpha + \smallO(1)$ and by Lemma~\ref{lem:typical_clique_number_existence} this must be the unique solution.

For the other direction, suppose that $\tcn = 1 + 1 / \alpha + \smallO(1)$ with $\alpha \in (0, 1)$. Then, by Definition~\ref{def:typical_clique_number} and Lemma~\ref{lem:typical_clique_number_bounded_additional},
\begin{align}
\tcn - 1
  &= \frac{\log\left(n\right) - \log\left(\tcn\right) + \log\left(\rm{\tcn-1}\right) + 1}{\log\left(\sf/\E[\wt]\right)}\\
  &= \frac{\log\left(n\right)}{\log\left(\sf/\E[\wt]\right)} + \smallO(1) \,.
\end{align}
Solving for $\sf$ we see that $\sf = n^{\alpha + \smallO(1)}$ since $\E[\wt]$ is uniformly bounded.
\end{proof}

\subsection{Proof of Theorem~\ref{thm:clique_number_two_point_concentration}: Concentration of the clique number}
\label{subsec:concentration_of_the_clique_number}
In this section we prove Theorem~\ref{thm:clique_number_two_point_concentration}, our main result. First we derive some useful results characterizing the relative moments. Then the proof itself is split into two parts:
the high-probability upper bound on the clique number in Sections \ref{subsubsec:upper_bound_on_the_clique_number} and \ref{subsubsec:upper_bound_on_the_clique_number2} using the first
moment method, and the high-probability lower bound on the clique number in Sections \ref{subsubsec:lower_bound_on_the_clique_number} and \ref{subsubsec:lower_bound_on_the_clique_number2}, using the
second moment method. In both parts we separately consider two cases: $\tcn \to \infty$ and $\tcn$ is bounded. In fact, a third might be possible, namely $\liminf_{n \to \infty} \tcn < \limsup_{n \to \infty} \tcn = \infty$. However, in that case we can apply the reasoning below to a maximal subsequence of $\tcn$ converging to infinity, and control the remaining terms by the argument used when $\tcn$ is bounded.

\subsubsection{Auxiliary results}
\label{subsubsec:auxiliary_results}
Binomial coefficients play an important role in counting the number of cliques. Therefore, it is crucial to have tight bounds on the binomial coefficient, provided by the lemma below. This lemma and the corresponding proof can be found in \cite{Spencer2014}:
\begin{lemma}
\label{lem:binomial_coefficient_stirling_approximation}
Suppose that $r = \smallO(\sqrt{n})$, then the binomial coefficient can be approximated by
\begin{equation}
\label{eq:binomial_coefficient_stirling_approximation}
\binom{n}{r} = (1 + \smallO(1)) \frac{1}{\sqrt{2 \pi r}} \left(\frac{n\,e}{r}\right)^r \,.
\end{equation}
\end{lemma}

Another important ingredient for the proof of Theorem~\ref{thm:clique_number_two_point_concentration} is to have sharp bounds on the relative moments from Definition~\ref{def:relative_moments}, which are provided by the following lemma. By definition, the typical clique number $\tcn$ is the solution to \eqref{eq:typical_clique_number}. If we consider the right-hand side and left-hand side of \eqref{eq:typical_clique_number} separately, then we see that these two functions must intersect at $\tcn$. Moreover, the right-hand side of \eqref{eq:typical_clique_number} will always grow more slowly than the left-hand side, as shown in the proof of Lemma~\ref{lem:typical_clique_number_existence}. This means that there exists a straight line in between these two functions, intersecting at $\tcn$ as illustrated in Figure~\ref{fig:relative_moments_bound}. Using this line we can then find bounds on the right-hand side of \eqref{eq:typical_clique_number} which in turn lead to bounds on the relative moments.
\begin{figure}[h]
    \centering
    \begin{tikzpicture}
    \begin{axis}[
        width=0.65\textwidth,
        height=0.40\textwidth,
        xlabel={},
        xmin=1, xmax=100,
        xtick={0,25,50,75,100},
        minor x tick num=1,
        ylabel={},
        ymin=0, ymax=100,
        ytick={0,25,50,75,100},
        minor y tick num=1,
        legend entries={$r$, $f_n(r)$, $\ell_n(r)$},
        legend cell align={left},
        legend style={
            legend pos=outer north east
            },
        ]
        \addplot[forget plot,white!80!black,thin,solid,domain=0:100,samples=4]{x};
        \addplot[black,thick,dotted,domain=0:100,samples=4]{x};

        \addplot[forget plot,white!80!black,thin,solid,domain=0:100,samples=200]{(x/20)^2+36};
        \addplot[black,thick,dashed,domain=0:100,samples=200]{(x/20)^2+36};

        \addplot[black,thick,solid,domain=0:100,samples=4]{(7/10)*(x-40)+40};

        \draw[->] (axis cs:31,49) to (axis cs:38,42);
        \node[] at (axis cs:27,53) {\small $\tcn$};
        \fill[black] (axis cs:40,40) circle[radius=1.2pt];
    \end{axis}
    \end{tikzpicture}
    \caption{Example of the line $\ell_n(r)$ with slope $\beta_n$ through $\tcn$ (solid line). Note that this line is a lower bound on $f_n(r)$ for all $r \in [1,\tcn]$, and an upper bound on $f_n(r)$ for all $r \in [\tcn,n]$.}
    \label{fig:relative_moments_bound}
\end{figure}
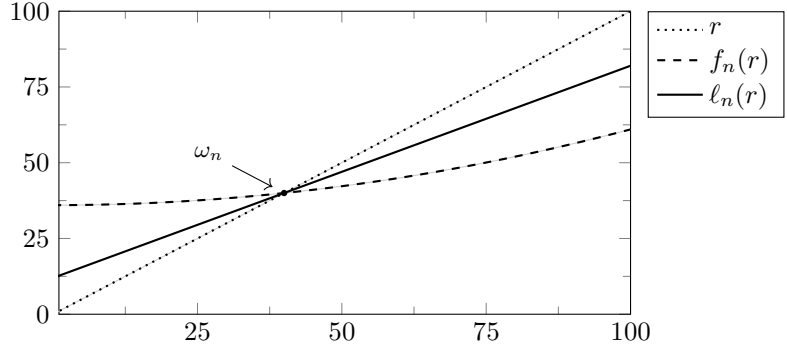

\begin{lemma}
\label{lem:relative_moments_bound}
Under Assumption \ref{ass:weights_probability_assumption}, the relative moments $\rm{r-1}$ from Definition~\ref{def:relative_moments} can be bounded by
\begin{align}
\rm{r-1} &\geq \left(\frac{\sf}{\E[\wt]}\right)^{\beta_n \bigl((r-1)-(\tcn-1)\bigr) + (\tcn-1)} \frac{r}{n \, e} \,,
  \qquad &&\text{for all } 1 \leq r \leq \tcn \,,\\
\rm{r-1} &\leq \left(\frac{\sf}{\E[\wt]}\right)^{\beta_n \bigl((r-1)-(\tcn-1)\bigr) + (\tcn-1)} \frac{r}{n \, e} \,,
  \qquad &&\text{for all } \tcn \leq r \leq n \,,
\intertext{with $\beta_n$ given by}
\label{eq:relative_moments_bound_beta}
\beta_n &= \frac{\log\left(\frac{\sf}{1 + \delta} \middle/ \E[\wt]\right)}{\log\left(\sf \middle/ \E[\wt]\right)} < 1\,,
\end{align}
and where $\delta > 0$ arises from Assumption~\ref{ass:weights_probability_assumption}.
\end{lemma}
\begin{proof}
Let $f_n(r)$ be as defined in \eqref{eq:right_hand_side_typical_clique_number}, then the typical clique number $\tcn$ is the solution in $r$ of $r = f_n(r)$. Consider $\ell_n(r) = \beta_n(r - \tcn) + \tcn$, which is the line through $\tcn$ with slope $\beta_n$, as shown in Figure~\ref{fig:relative_moments_bound}. We will show that, for all $n$, the line $\ell_n(r)$ is a lower bound on $f_n(r)$ when $r \in [1,\tcn]$, and an upper bound on $f_n(r)$ when $r \in [\tcn,n]$.

The slope of $f_n(r)$ was derived in \eqref{eq:right_hand_side_typical_clique_number_derivative} and is bounded by $\beta_n$ given in \eqref{eq:relative_moments_bound_beta}. Hence, we have $\ell_n(r) \leq f_n(r)$ when $r \in [1,\tcn]$ and $\ell_n(r) \geq f_n(r)$ otherwise.

To finish the proof note that, for all $r \in [1,\tcn]$,
\begin{align}
\beta_n \bigl((r-1) - (\tcn-1)\bigr) + \tcn
  &= \ell_n(r)\\
  &\leq f_n(r)
  = \frac{\log(n) - \log(r) + \log(\rm{r-1}) + 1}{\log(\sf / \E[\wt])} + 1 \,.
\end{align}
Exponentiating both sides yields
\begin{equation}
\left(\frac{\sf}{\E[\wt]}\right)^{\beta_n \bigl((r-1) - (\tcn-1)\bigr) + (\tcn - 1)}
  \leq \frac{n \, \rm{r-1} \, e}{r} \,.
\end{equation}
Multiplying both sides by $r / (n \, e)$ gives the first result. The second result follows similarly.
\end{proof}

\subsubsection{Upper bound with diverging typical clique number}
\label{subsubsec:upper_bound_on_the_clique_number}
In this section we prove the upper bound of Theorem~\ref{thm:clique_number_two_point_concentration} assuming that $\tcn \to \infty$. Define the event
\begin{equation}
\label{eq:truncation_event}
\te = \left\{\max_{i \in [n]} \w_i \leq \frac{\sf}{1 + \delta}\right\} \,.
\end{equation}
Assumption~\ref{ass:weights_probability_assumption} enforces that $\P(\te) \to 1$ as $n \to \infty$. A trivial, but crucial, observation is that the joint distribution of the weights $(\w_1,\ldots,\w_n)$ conditional on the event $\te$ is the same as that of a sequence of i.i.d.\ truncated weights $\smash{(\wt_1,\ldots,\wt_n)}$. In other words
\begin{equation}
\label{eq:conditional_weight_distribution}
(\w_1,\ldots,\w_n)\,|\,\te \stackrel{d}{=} (\wt_1,\ldots,\wt_n) \,,
\end{equation}
where $\wt_i$ are i.i.d.\ random variables with the same distribution as $\wt$. This statement can be checked by an elementary computation.

Let $\cn(G_n)$ be the clique number of the graph $G_n$ and define $N_r$ to be the number of cliques of size $r$ in $G_n$. Then, by Assumption~\ref{ass:weights_probability_assumption} and the first moment method,
\begin{align}
\P(\cn(G_n) \geq r)
  &= (1 + \smallO(1)) \P(\cn(G_n) \geq r \,|\, \te)\\[1ex]
  &= (1 + \smallO(1)) \P(N_r \geq 1 \,|\, \te)\\[1ex]
\label{eq:first_moment_method}
  &\leq (1 + \smallO(1)) \E[N_r \,|\, \te] \,.
\end{align}
Then by linearity of expectation,
\begin{align}
\E[N_r \,|\, \te]
  &= \sum_{C\subseteq [n]:\,|C| = r} \P(C \text{ is a clique in } G_n \,|\, \te)\\
  &= \sum_{C\subseteq [n]:\,|C| = r} \E\left[\prod_{i<j \in C} \frac{\w_i}{\sf} \cdot \frac{\w_j}{\sf} \wedge 1 \,\middle|\, \te \right]\\
  &= \sum_{C\subseteq [n]:\,|C| = r} \E\left[\prod_{i<j \in C} \frac{\wt_i}{\sf} \cdot \frac{\wt_j}{\sf}\right]\\
\label{eq:expected_cliques_of_size}
  &= \binom{n}{r} \E\left[\left(\frac{\wt}{\sf}\right)^{r-1}\right]^r\\
  &= \binom{n}{r} \left(\rm{r-1} \left(\frac{\E[\wt]}{\sf}\right)^{r-1}\right)^r \,.
\end{align}
To prove the upper bound of Theorem~\ref{thm:clique_number_two_point_concentration}, we need show that $\E[N_r \,|\, \te] \to 0$, as $n \to \infty$, when $r > \lfloor \tcn + \epsilon \rfloor$, and since $r$ is integer this implies $r \geq \tcn + \epsilon$. Using Lemmas~\ref{lem:binomial_coefficient_stirling_approximation} and \ref{lem:relative_moments_bound} we can further bound the above expression as
\begin{align}
\E[N_r \,|\, \te]
  &\leq (1 + \smallO(1)) \frac{1}{\sqrt{2 \pi r}} \left(\frac{n\,e}{r}\right)^r \left(\rm{r-1} \left(\frac{\E[\wt]}{\sf}\right)^{r-1}\right)^r\\
  &\leq (1 + \smallO(1)) \frac{1}{\sqrt{2 \pi r}} \left(\frac{\E[\wt]}{\sf}\right)^{r(r-1)-r(\beta_n ((r - 1) - (\tcn - 1)) + (\tcn - 1))}\\
\label{eq:expected_cliques_of_size_bound}
  &= (1 + \smallO(1)) \frac{1}{\sqrt{2 \pi r}} \left(\frac{\E[\wt]}{\sf}\right)^{-(1-\beta_n) \cdot r \bigl((\tcn - 1) - (r - 1)\bigr)}\,,
\end{align}
where $\beta_n$ comes from Lemma~\ref{lem:relative_moments_bound}. Using the definition of $\beta_n$ we have
\begin{equation}
\label{eq:maximal_slope_bound_applied}
\left(\frac{\sf}{\E[\wt]}\right)^{-(1-\beta_n)}
  = \left(\frac{\E[\wt]}{\sf}\right) \left(\frac{\sf / (1 + \delta)}{\E[\wt]}\right)
  = \frac{1}{1 + \delta} \,.
\end{equation}

Combining \eqref{eq:expected_cliques_of_size_bound} and \eqref{eq:maximal_slope_bound_applied} and because $r - \tcn \geq \epsilon$ we obtain
\begin{align}
\E[N_r \,|\, \te]
  &= (1 + \smallO(1)) \frac{1}{\sqrt{2 \pi r}} \left(\frac{1}{1 + \delta}\right)^{- r \bigl(\tcn - r\bigr)}\\
\label{eq:expected_cliques_of_size_bound_further}
  &\leq (1 + \smallO(1)) \frac{1}{\sqrt{2 \pi r}} \left(\frac{1}{1 + \delta}\right)^{r \, \epsilon} \,.
\end{align}

Since $\tcn \to \infty$ it is easily seen from \eqref{eq:expected_cliques_of_size_bound_further} that $\E[N_r \,|\, \te] \to 0$ when $r > \lfloor\tcn + \epsilon\rfloor$. Hence it follows from \eqref{eq:first_moment_method} that $\P(\cn(G_n) > \lfloor \tcn + \epsilon \rfloor) \to 0$. \hfill\qedsymbol

\subsubsection{Upper bound with bounded typical clique number}
\label{subsubsec:upper_bound_on_the_clique_number2}
In this section we prove the upper bound of Theorem~\ref{thm:clique_number_two_point_concentration} assuming that $\tcn$ is bounded. First we consider the case where $\tcn$ converges, in this case there exists an $\alpha > 0$ such that $\tcn = 1 / \alpha + 1 + \smallO(1)$. We want to apply all the steps in Section~\ref{subsubsec:upper_bound_on_the_clique_number}, but instead of conditioning on the event in \eqref{eq:truncation_event} we will condition on the event
\begin{equation}
\label{eq:truncation_event_increasing_delta}
\teinc = \left\{\max_{i \in [n]} \w_i \leq \frac{\sf}{1 + \eta}\right\} \,,
\end{equation}
where $\eta > 0$ comes from Assumption~\ref{ass:weights_probability_assumption_increasing_delta}.

Since $\tcn = 1 / \alpha + 1 + \smallO(1)$ it follows from Lemma~\ref{lem:typical_clique_number_bounded} that $\sf = n^{\alpha + \smallO(1)}$, and by Assumption~\ref{ass:weights_probability_assumption_increasing_delta} we have $\P(\teinc) \to 1$. Moreover, by repeating the steps in Lemmas \ref{lem:typical_clique_number_bounded} and \ref{lem:typical_clique_number_bounded_additional} it can easily be checked that replacing $\delta$ by $\eta$ in Definitions \ref{def:relative_moments} and \ref{def:typical_clique_number} the typical clique number remains equal to $\tcn = 1 / \alpha + 1 + \smallO(1)$. Therefore, we can follow all steps in Section~\ref{subsubsec:upper_bound_on_the_clique_number} but conditioning on $\teinc$ instead of $\te$. This gives
\begin{align}
\P(\cn(G_n) \geq r)
  &= (1 + \smallO(1)) \P(\cn(G_n) \geq r \,|\, \teinc)\\[1ex]
  &\leq (1 + \smallO(1)) \E[N_r \,|\, \teinc]\\
\label{eq:expected_cliques_of_size_bound2}
  &\leq (1 + \smallO(1)) \frac{1}{\sqrt{2 \pi r}} \left(\frac{1}{1 + \eta}\right)^{r \, \epsilon} \,.
\end{align}
Since $\eta > 0$ is arbitrary and $r = \tcn + \epsilon$ is bounded it follows from \eqref{eq:expected_cliques_of_size_bound2} that we can make $\P(\cn(G_n) \geq r)$ arbitrarily small, hence $\P(\cn(G_n) \geq r) \to 0$.

To complete the proof we consider the case when $\tcn$ does not converge. In this case, we know that every subsequence $\smash{\cramped{(n_i)_{i \in \mathbb{N}}}}$ contains a further subsequence $\smash{\cramped{(n_{i_j})_{j \in \mathbb{N}}}}$ along which $\smash{\cramped{\tcn[n_{i_j}]}}$ converges. Applying the arguments above shows that every subsequence $\smash{\cramped{(n_i)_{i \in \mathbb{N}}}}$ has a further subsequence $\smash{\cramped{(n_{i_j})_{j \in \mathbb{N}}}}$ along which $\P(\smash{\cn(\cramped{G_{n_{i_j}}})} \geq r) \to 0$, and it follows that $\P(\cn(G_n) \geq r) \to 0$. \hfill\qedsymbol

\subsubsection{Lower bound with diverging typical clique number}
\label{subsubsec:lower_bound_on_the_clique_number}
In this section we prove the lower bound of Theorem~\ref{thm:clique_number_two_point_concentration} assuming that $\tcn \to \infty$. Recall that $\cn(G_n)$ denotes the clique number of the graph $G_n$ and $N_r$ is the number of cliques of size $r$ in $G_n$. Then, by the second moment method, and using the truncation event $\te$ given by \eqref{eq:truncation_event} together with Assumption~\ref{ass:weights_probability_assumption},
\begin{align}
\P(\cn(G_n) < r)
  &= (1 + \smallO(1)) \P(\cn(G_n) < r \,|\, \te)\\[1ex]
  &= (1 + \smallO(1)) \P(N_r = 0 \,|\, \te)\\[1ex]
  &\leq (1 + \smallO(1)) \frac{\textup{Var}(N_r \,|\, \te)}{\E[N_r \,|\, \te]^2}\\
\label{eq:second_moment_method}
  &= (1 + \smallO(1)) \left(\frac{\E[N_r^2 \,|\, \te]}{\E[N_r \,|\, \te]^2} - 1\right) \,.
\end{align}
Hence we need to show that $\E[N_r^2 \,|\, \te] / \E[N_r \,|\, \te]^2 \to 1$ as $n \to \infty$, with $r = \lfloor \tcn - \epsilon \rfloor$.
The first moment of the number of cliques $N_r$ was computed in \eqref{eq:expected_cliques_of_size}, and is given by
\begin{equation}
\label{eq:expected_cliques_of_size2}
\E[N_r \,|\, \te]
  = \binom{n}{r} \E\left[\left(\frac{\wt}{\sf}\right)^{r-1}\right]^r \,.
\end{equation}
Similarly, the second moment of the number of cliques $N_r$ is also found using \eqref{eq:conditional_weight_distribution} and linearity of expectation as
\begin{align}
\hspace{20pt}&\hspace{-20pt}
\E[N_r^2 \,|\, \te]
  = \sum_{|C_1| = r,\,|C_2| = r} \P(C_1 \text{ and } C_2 \text{ are both cliques in } G_n \,|\, \te)\\
  &= \sum_{|C_1| = r,\,|C_2| = r} \E\left[\frac{\prod_{i<j \in C_1} \frac{\wt_i}{\sf} \cdot \frac{\wt_j}{\lambda_n} \prod_{i<j \in C_2} \frac{\wt_i}{\sf} \cdot \frac{\wt_j}{\sf}}{\prod_{i<j \in C_1 \cap C_2} \frac{\wt_i}{\sf} \cdot \frac{\wt_j}{\sf}}\right]\\
\label{eq:expected_double_cliques_of_size}
  &= \sum_{k=0}^r\;\sum_{\substack{|C_1| = r,\,|C_2| = r,\\|C_1 \cap C_2| = k}} \E\left[\left(\frac{\wt}{\sf}\right)^{r-1}\right]^{2(r-k)} \E\left[\left(\frac{\wt}{\sf}\right)^{2(r-1)-(k-1)}\right]^k\\
\label{eq:expected_double_cliques_of_size2}
  &= \sum_{k=0}^r \binom{n}{r}\binom{r}{k}\binom{n-r}{r-k} \, \E\left[\left(\frac{\wt}{\sf}\right)^{r-1}\right]^{2(r-k)} \E\left[\left(\frac{\wt}{\sf}\right)^{2(r-1)-(k-1)}\right]^k \,.
\end{align}
The third equality comes from counting how many times each vertex is an endpoint of an edge, and thus how many times each weight is present in the product. We count two cases separately:
\begin{itemize}
  \item Vertices in $C_1 \setminus C_2$ will need edges to each other vertex in $C_1$. So, each vertex in $C_1 \setminus C_2$ will be $r-1$ times in the product of \eqref{eq:expected_double_cliques_of_size} and similarly for vertices in $C_2 \setminus C_1$. Since we have $2(r - k)$ vertices in $C_1 \setminus C_2$ and $C_2 \setminus C_1$ we get the $\E[W^{r-1}]^{2(r-k)}$ term. See Figure~\ref{fig:double_clique_probability_left}.
  \item Vertices in $C_1 \cap C_2$ will need edges to each vertex in $C_1 \cup C_2$. So, each vertex in $C_1 \cap C_2$ will be $2(r-1)-(k-1)$ times in the product of \eqref{eq:expected_double_cliques_of_size} and we have $k$ vertices in $C_1 \cap C_2$. So we get the $\E[W^{2(r-1)-(k-1)}]^k$ term. See Figure~\ref{fig:double_clique_probability_right}.
\end{itemize}
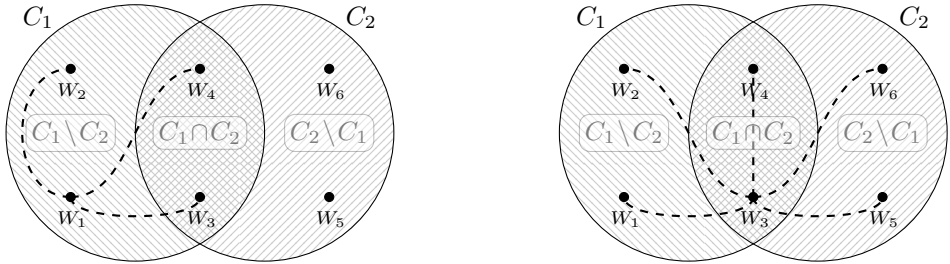
\begin{figure}[H]
\centering
\begin{subfigure}[t]{0.48\textwidth}
\centering
\begin{tikzpicture}[scale=0.85]
    \fill[pattern=north west lines,pattern color=white!80!black] (1,1) circle [radius=2];
    \fill[pattern=north east lines,pattern color=white!80!black] (3,1) circle [radius=2];
    \draw (1,1) circle [radius=2];
    \draw (3,1) circle [radius=2];
    \node at (-0.5,2.8) {$C_1$};
    \node at ( 4.5,2.8) {$C_2$};

    \draw[black,thick,dashed] (0,0) .. controls (-1,0) and (-1,2) .. (0,2);
    \draw[black,thick,dashed] (0,0) .. controls (0,-0.4) and (2,-0.4) .. (2,0);
    \draw[black,thick,dashed] (0,0) .. controls (1,0) and (1,2) .. (2,2);

    \filldraw (0,0) circle (2pt) node[align=left, below=2pt] {\scriptsize$\,\w_1$};
    \filldraw (0,2) circle (2pt) node[align=left, below=2pt] {\scriptsize$\,\w_2$};
    \filldraw (2,0) circle (2pt) node[align=left, below=2pt] {\scriptsize$\,\w_3$};
    \filldraw (2,2) circle (2pt) node[align=left, below=2pt] {\scriptsize$\,\w_4$};
    \filldraw (4,0) circle (2pt) node[align=left, below=2pt] {\scriptsize$\,\w_5$};
    \filldraw (4,2) circle (2pt) node[align=left, below=2pt] {\scriptsize$\,\w_6$};

    \node[fill=white,fill opacity=0.5,draw opacity=0.25,text opacity=1.0,draw,rounded corners,inner sep=2pt] at (0,1) {$C_1\!\setminus\!C_2$};
    \node[fill=white,fill opacity=0.5,draw opacity=0.25,text opacity=1.0,draw,rounded corners,inner sep=2pt] at (2,1) {$C_1\!\cap\!C_2$};
    \node[fill=white,fill opacity=0.5,draw opacity=0.25,text opacity=1.0,draw,rounded corners,inner sep=2pt] at (4,1) {$C_2\!\setminus\!C_1$};
\end{tikzpicture}
\caption{Example of all the edges connecting to a vertex in $C_1 \setminus C_2$.}
   \label{fig:double_clique_probability_left}
\end{subfigure}%
\hspace{0.04\textwidth}%
\begin{subfigure}[t]{0.48\textwidth}
\centering
\begin{tikzpicture}[scale=0.85]
    \fill[pattern=north west lines,pattern color=white!80!black] (1,1) circle [radius=2];
    \fill[pattern=north east lines,pattern color=white!80!black] (3,1) circle [radius=2];
    \draw (1,1) circle [radius=2];
    \draw (3,1) circle [radius=2];
    \node at (-0.5,2.8) {$C_1$};
    \node at ( 4.5,2.8) {$C_2$};

    \draw[black,thick,dashed] (2,0) to (2,2);
    \draw[black,thick,dashed] (2,0) .. controls (2,-0.4) and (0,-0.4) .. (0,0);
    \draw[black,thick,dashed] (2,0) .. controls (2,-0.4) and (4,-0.4) .. (4,0);
    \draw[black,thick,dashed] (2,0) .. controls (1,0) and (1,2) .. (0,2);
    \draw[black,thick,dashed] (2,0) .. controls (3,0) and (3,2) .. (4,2);

    \filldraw (0,0) circle (2pt) node[align=left, below=2pt] {\scriptsize$\,\w_1$};
    \filldraw (0,2) circle (2pt) node[align=left, below=2pt] {\scriptsize$\,\w_2$};
    \filldraw (2,0) circle (2pt) node[align=left, below=2pt] {\scriptsize$\,\w_3$};
    \filldraw (2,2) circle (2pt) node[align=left, below=2pt] {\scriptsize$\,\w_4$};
    \filldraw (4,0) circle (2pt) node[align=left, below=2pt] {\scriptsize$\,\w_5$};
    \filldraw (4,2) circle (2pt) node[align=left, below=2pt] {\scriptsize$\,\w_6$};

    \node[fill=white,fill opacity=0.5,draw opacity=0.25,text opacity=1.0,draw,rounded corners,inner sep=2pt] at (0,1) {$C_1\!\setminus\!C_2$};
    \node[fill=white,fill opacity=0.5,draw opacity=0.25,text opacity=1.0,draw,rounded corners,inner sep=2pt] at (2,1) {$C_1\!\cap\!C_2$};
    \node[fill=white,fill opacity=0.5,draw opacity=0.25,text opacity=1.0,draw,rounded corners,inner sep=2pt] at (4,1) {$C_2\!\setminus\!C_1$};
\end{tikzpicture}
\caption{Example of all the edges connecting to a vertex in $C_1 \cap C_2$.}
\label{fig:double_clique_probability_right}
\end{subfigure}%
\vspace*{-1.5ex}
\caption{Example of edges connecting to vertices in different parts of $C_1 \cup C_2$.}
\label{fig:double_clique_probability}
\end{figure}
Combining \eqref{eq:expected_cliques_of_size2} and \eqref{eq:expected_double_cliques_of_size2} we obtain
\begin{align}
\frac{\E\left[N_r^2 \,|\, \te\right]}{\E[N_r \,|\, \te]^2}
  &= \sum_{k = 0}^r \frac{\binom{r}{k}\binom{n-r}{r-k}}{\binom{n}{r}} \cdot \frac{\E\left[\left(\frac{\wt}{\sf}\right)^{r-1}\right]^{2(r-k)} \E\left[\left(\frac{\wt}{\sf}\right)^{2(r-1)-(k-1)}\right]^{k}}{\E\left[\left(\frac{\wt}{\sf}\right)^{r-1}\right]^{2r}}\\
  &= \sum_{k = 0}^r \frac{\binom{r}{k}\binom{n-r}{r-k}}{\binom{n}{r}} \left(\frac{\rm{2(r-1)-(k-1)}}{\rm{r-1}^2}\right)^k
  \left(\frac{\E[\wt]}{\sf}\right)^{-k(k-1)}\\
  &\leq 1 + \sum_{k = 1}^r \frac{\binom{r}{k}\binom{n-r}{r-k}}{\binom{n}{r}} \left(\frac{\rm{2(r-1)-(k-1)}}{\rm{r-1}^2}\right)^k
  \left(\frac{\E[\wt]}{\sf}\right)^{-k(k-1)}\\
\label{eq:second_moment_bound}
  &\leq 1 + \max_{1 \leq k \leq r}
  \underbrace{\,
  r\,\frac{\binom{r}{k}\binom{n-r}{r-k}}{\binom{n}{r}} \left(\frac{\rm{2(r-1)-(k-1)}}{\rm{r-1}^2}\right)^k
  \left(\frac{\E[\wt]}{\sf}\right)^{-k(k-1)}\,}_{\displaystyle \coloneqq b_k^{n,r}} .
\end{align}
We will show that $\max_{k \in [r]} b_k^{n,r} \to 0$ as $n \to \infty$. To continue we consider two cases: \emph{(i)} $k = r$; and \emph{(ii)} $1 \leq k \leq r-1$.

\emph{Case (i):} Here $k = r$, so we want to show that $b_r^{n,r} \to 0$ as $n \to \infty$. By definition of $b_r^{n,r}$ and by Lemmas \ref{lem:binomial_coefficient_stirling_approximation} and \ref{lem:relative_moments_bound},
\begin{align}
b_r^{n,r}
  &= r\,\frac{1}{\binom{n}{r}} \left(\frac{1}{\rm{r-1}}\right)^r\left(\frac{\sf}{\E[\wt]}\right)^{r(r-1)}\\
  &\leq (1 + \smallO(1))\,r\sqrt{2 \pi r}\left(\frac{r}{ne}\right)^r \left(\frac{1}{\rm{r-1}}\right)^r\left(\frac{\sf}{\E[\wt]}\right)^{r(r-1)}\\
  &\leq (1 + \smallO(1))\,r\sqrt{2 \pi r} \left(\frac{\sf}{\E[\wt]}\right)^{r\bigl((r-1) - \beta_n\bigl((r-1)-(\tcn-1)\bigr) - (\tcn - 1) \bigr)}\\
  &= (1 + \smallO(1))\,r\sqrt{2 \pi r} \left(\frac{\sf}{\E[\wt]}\right)^{-(1 - \beta_n) r\bigl((\tcn-1) - (r-1)\bigr)} \,.
\end{align}
where $\beta_n$ is given in Lemma~\ref{lem:relative_moments_bound}.

Using \eqref{eq:maximal_slope_bound_applied} together with the fact that $r = \lfloor \tcn - \epsilon \rfloor \leq \tcn - \epsilon$ yields the bound
\begin{align}
b_r^{n,r}
  &\leq (1 + \smallO(1))\,r\sqrt{2 \pi r} \left(\frac{\sf}{\E[\wt]}\right)^{-(1 - \beta) r\bigl((\tcn-1) - (r-1)\bigr)}\\
\label{eq:bounded_br2}
  &\leq (1 + \smallO(1))\,r\sqrt{2 \pi r} \left(\frac{1}{1 + \delta}\right)^{\epsilon r} .
\end{align}
Since $\tcn \to \infty$ it is easily seen from \eqref{eq:bounded_br2} that $b_r^{n,r} \to 0$.

\emph{Case (ii):} Here we must show that $\max_{k \in [r-1]} b_k^{n,r} \to 0$ as $n \to \infty$. First we apply Lemma~\ref{lem:binomial_coefficient_stirling_approximation} on the binomial coefficients, which gives
\begin{align}
\frac{\binom{r}{k}\binom{n-r}{r-k}}{\binom{n}{r}}
  &= (1 + \smallO(1)) \sqrt{\frac{r}{2\pi k(r - k)}} \left(\frac{r\,e}{k}\right)^k \left(\frac{(n-r)\,e}{r-k}\right)^{r-k} \left(\frac{n\,e}{r}\right)^{-r}\\
  &= (1 + \smallO(1)) \sqrt{\frac{r}{2\pi k(r - k)}} \left(\frac{r\,e}{k}\right)^k \left(\frac{n-r}{n}\,\frac{r}{r-k}\right)^r \left(\frac{r-k}{(n-r)\,e}\right)^k\\
\label{eq:hypergeometric_density_bound}
  &= (1 + \smallO(1)) \sqrt{\frac{r}{2\pi k(r - k)}} \left(\frac{r\,e}{k}\right)^k \left(\frac{r-k}{n-r}\right)^k .
\end{align}
Now, for all $1 \leq k \leq r - 1$ we have that $k \leq \tcn - \epsilon - 1 \leq \tcn - 2\epsilon$ and therefore $2(r-1)-(k-1) \geq \tcn - 1$. So, we can apply Lemma~\ref{lem:relative_moments_bound} on both $\rm{2(r-1)-(k-1)}$ and on $\rm{r-1}$, yielding
\begin{align}
\hspace{55pt}&\hspace{-55pt}
\left(\frac{\rm{2(r-1)-(k-1)}}{\rm{r-1}^2}\right)^k
\left(\frac{\sf}{\E[\wt]}\right)^{k(k-1)}\\
  &\leq \left(\frac{2(r-1)-(k-1)+1}{n\,e}\right)^k \left(\frac{\vphantom{(}n\,e}{r}\right)^{2k} \left(\frac{\sf}{\E[\wt]}\right)^{k(k-1)}\\
  &\qquad{\scriptstyle\times}\,
  \left(\frac{\sf}{\E[\wt]}\right)^{k\bigl(\beta_n(2(r-1)-(k-1)-(\tcn-1)) + (\tcn - 1)\bigr)}\\
  &\qquad{\scriptstyle\times}\,
  \left(\frac{\sf}{\E[\wt]}\right)^{-2k\bigl(\beta_n((r-1)-(\tcn-1)) + (\tcn-1)\bigr)}\\
\label{eq:maximal_slope_bound_applied_extended}
  &= \left(\frac{2(r-1)-(k-1)+1}{r^2}\,ne\right)^k\,\left(\frac{\sf}{\E[\wt]}\right)^{-(1 - \beta_n) \cdot k\bigl((\tcn-1)-(k-1)\bigr)} \,. \quad
\end{align}
Combining \eqref{eq:hypergeometric_density_bound} and \eqref{eq:maximal_slope_bound_applied_extended} we obtain
\begin{align}
b_k^{n,r}
  &\leq (1 + \smallO(1)) \, r \sqrt{\frac{r}{2\pi k(r - k)}} \left(\frac{r-k}{k} \, \frac{2(r-1)-(k-1)+1}{r} \, e^2\right)^k\\
  &\qquad{\scriptstyle\times}\,
  \left(\frac{\sf}{\E[\wt]}\right)^{-(1 - \beta_n)\cdot k\bigl((\tcn-1)-(k-1)\bigr)}\\
  &\leq (1 + \smallO(1)) \, r \sqrt{\frac{r}{2\pi k(r - k)}} \left(\frac{r-k}{k} \, 2 e^2\right)^{k} \left(\frac{\sf}{\E[\wt]}\right)^{\!-(1 - \beta_n) \cdot k\bigl((\tcn-1)-(k-1)\bigr)} .
\end{align}
Here we can use \eqref{eq:maximal_slope_bound_applied} again. This gives
\begin{align}
b_k^{n,r}
  &\leq (1 + \smallO(1)) \, r \sqrt{\frac{r}{2\pi k(r - k)}} \left(\frac{r-k}{k} \, 2 e^2\right)^{k} \left(\frac{\sf}{\E[\wt]}\right)^{-(1 - \beta_n) \cdot k\bigl((\tcn-1)-(k-1)\bigr)}\\
  &\leq (1 + \smallO(1)) \, r \sqrt{\frac{r}{2\pi k(r - k)}} \left(\frac{r-k}{k} \, 2 e^2\right)^{k} \left(\frac{1}{1 + \delta}\right)^{k\bigl((\tcn-1)-(k-1)\bigr)}\\
\label{eq:bounded_bk2}
  &= (1 + \smallO(1)) \, r \sqrt{\frac{r}{2\pi k(r - k)}} \left(2 e^2 \left(\frac{r}{k}-1\right)\left(\frac{1}{1 + \delta}\right)^{\tcn-k}\right)^{k} .
\end{align}
Fix $\zeta \in \smash{\bigl(0, (1 + 2e^2)^{-1}\bigr)}$ and recall that $\tcn \to \infty$. Then it can easily be seen from \eqref{eq:bounded_bk2} that $\max_{1 \leq k \leq (1-\zeta)r} b_k^{n,r} \to 0$ since $(1 + \delta)^{-(\tcn - k)} \to 0$ exponentially, eventually dominating the other terms. Finally, to show that $\max_{(1-\zeta)r \leq k \leq r} b_k^{n,r} \to 0$, note that $2 e^2 (r / k - 1) < 1$ and therefore $\bigl(2 e^2 (r / k-1)(1 + \delta)^{-(\tcn-k)}\bigr)^{k} \to 0$ exponentially, again dominating the remaining terms.

Hence $\max_{k \in [r]} b_k^{n,r} \to 0$ as $n \to \infty$ and $r = \lfloor \tcn - \epsilon \rfloor$. Using \eqref{eq:second_moment_bound} and \eqref{eq:second_moment_method} it follows that $\P(\cn(G_n) < \lfloor \tcn - \epsilon \rfloor) \to 0$ as $n \to \infty$. \hfill\qedsymbol

\subsubsection{Lower bound with bounded typical clique number}
\label{subsubsec:lower_bound_on_the_clique_number2}
In this section we prove the lower bound of Theorem~\ref{thm:clique_number_two_point_concentration} assuming that $\tcn$ is bounded. First we consider the case where $\tcn$ converges, in this case there exists an $\alpha > 0$ such that $\tcn = 1 / \alpha + 1 + \smallO(1)$. We want to apply all the steps in Section~\ref{subsubsec:lower_bound_on_the_clique_number}, but instead of conditioning on the event $\te$ given in \eqref{eq:truncation_event} we will condition on the event $\teinc$ given in \eqref{eq:truncation_event_increasing_delta}. As shown in Section~\ref{subsubsec:upper_bound_on_the_clique_number2}, the typical clique number $\tcn$ is unaffected by this change, and by Assumption~\ref{ass:weights_probability_assumption_increasing_delta} we also have $\P(\teinc) \to 1$.

Now, following all steps in Section~\ref{subsubsec:upper_bound_on_the_clique_number} but conditioning on $\teinc$ instead of $\te$ we obtain
\begin{align}
\P(\cn(G_n) < r)
  &= (1 + \smallO(1)) \P(\cn(G_n) < r \,|\, \teinc)
  \leq (1 + \smallO(1)) \frac{\E[N_r^2 \,|\, \teinc]}{\E[N_r \,|\, \teinc]^2} - 1 \,.\\
\intertext{By combining \eqref{eq:bounded_br2} and \eqref{eq:bounded_bk2}, and using that $r = \tcn - \epsilon$ is bounded we get}
\P(\cn(G_n) < r)
  &\leq (1 + \smallO(1)) \frac{\E[N_r^2 \,|\, \teinc]}{\E[N_r \,|\, \teinc]^2} - 1
  \leq \bigO(1) \, \left(\frac{1}{1 + \eta}\right)^{\epsilon} \,.
\end{align}
Since we can make $\eta > 0$ arbitrarily large it follows that $\P(\cn(G_n) < r) \to 0$.

To complete the proof we consider the case when $\tcn$ does not converge. In this case, we know that every subsequence $\smash{\cramped{(n_i)_{i \in \mathbb{N}}}}$ contains a further subsequence $\smash{\cramped{(n_{i_j})_{j \in \mathbb{N}}}}$ along which $\smash{\cramped{\tcn[n_{i_j}]}}$ converges. Applying the arguments above shows that every subsequence $\smash{\cramped{(n_i)_{i \in \mathbb{N}}}}$ has a further subsequence $\smash{\cramped{(n_{i_j})_{j \in \mathbb{N}}}}$ along which $\P(\smash{\cn(\cramped{G_{n_{i_j}}})} < r) \to 0$, and it follows that $\P(\cn(G_n) < r) \to 0$. \hfill\qedsymbol

\appendix
\section{Derivation of examples}
\label{sec:derivation_of_examples}
In this section we derive the asymptotic behavior of the typical clique number $\tcn$ for some given weight distributions $\w$ and scalings $\sf$. This can be very difficult in general, but for several choices of weights good asymptotic characterizations can be given. An overview of these results can be found in Tables \ref{tbl:typical_clique_number_short_tailed}, \ref{tbl:typical_clique_number_light_tailed1}, and \ref{tbl:typical_clique_number_light_tailed2}.

Throughout the derivation of the examples below we make use of the Lambert-W functions, which are obtained from the solutions in $y \in \mathbb{R}$ of the equation
\begin{equation}
\label{eq:lambert_w}
x = y \, e^y \,,
\end{equation}
\noeqref{eq:lambert_w}%
When $x \geq 0$ this has a unique real solution, while for $x \in (-1/e, 0)$ there are two real solutions. This gives rise to two branches: the \emph{principal branch}, denoted by $\lambertW_0 : [-1/e,\infty) \mapsto [-1,\infty)$ and the \emph{lower branch}, denoted by $\lambertW_{-1} : [-1/e,0) \mapsto (-\infty,-1]$. For an overview of this function and its properties see \cite{Corless1996}.

\subsection{Bernoulli weights}
\label{subsec:bernoulli_weights}
Let $\w$ have a Bernoulli distribution with parameter $p$, that is $\w \sim \text{Ber}(p)$, and take any scaling $\sf \geq c > 1$. In this case, we have an Erd\H{o}s--R\'{e}nyi random graph with connection probability $\sf^{-2}$ on approximately $n p$ vertices, with all remaining vertices being isolated. Therefore, by \eqref{eq:typical_clique_number_erdos_renyi}, we expect the typical clique number $\tcn$ to be
\begin{equation}
\tcn = \log_{\sf}(n\,p) - \log_{\sf}\log_{\sf}(n\,p) + \log_{\sf}(e) + 1 + \smallO(1)\,.
\end{equation}
In this section, we show that the same result is obtained by solving \eqref{eq:typical_clique_number} from Definition~\ref{def:typical_clique_number}.

The relative moments from Definition~\ref{def:relative_moments} are given by
\begin{equation}
\rm{r-1}
  = \frac{\E\bigl[\wt^{r-1}\bigr]}{\E\bigl[\wt\bigr]^{r-1}}
  = \frac{\E\bigl[\w^{r-1}\bigr]}{\E\bigl[\w\bigr]^{r-1}}
  = p^{2-r}\,.
\end{equation}
The typical clique number $\tcn$ from Definition~\ref{def:typical_clique_number} is given by the solution in $r$ of
\begin{equation}
r = \frac{\log(n) - \log(r) + (2-r)\log(p) + 1}{\log\bigl(\sf \,/\, p\bigr)} + 1\,.
\end{equation}
Solving this we obtain
\begin{equation}
\tcn = \frac{\lambertW_0(n\,p\,e\,\sf\log(\sf))}{\log(\sf)}\,.
\end{equation}
where $\lambertW_0$ denotes the principal branch of the Lambert-W function, see \eqref{eq:lambert_w}. We can simplify the solution above using the approximation $\lambertW_0(x) = \log(x) - \log\log(x) + \smallO(1)$ as $x \to \infty$ from \cite{Corless1996}. 
This gives
\begin{align}
\tcn
  &= \frac{\log(n\,p\,e\,\sf\log(\sf)) - \log\log(n\,p\,e\,(\sf)\log(\sf))}{\log(\sf)} + \smallO(1)\\
  &= \frac{\log(n\,p\,e\,\sf\log(\sf)) - \log\log(n\,p)}{\log(\sf)} + \smallO(1)\\
  &= \log_{\sf}(n\,p\,\log(\sf)) - \log_{\sf}\log(n\,p) + \log_{\sf}(e) + 1 + \smallO(1)\\[0.5ex]
  &= \log_{\sf}(n\,p) - \log_{\sf}\log_{\sf}(n\,p) + \log_{\sf}(e) + 1 + \smallO(1)\,,
\end{align}
which is exactly the expected solution.

\subsection{Beta weights}
\label{subsec:beta_weights}
Let $\w$ have a beta distribution with parameters $\alpha > 0$ and $\beta > 0$, that is $\w \sim \text{Beta}(\alpha, \beta)$, and take any scaling $\sf \geq c > 1$. Then the relative moments from Definition~\ref{def:relative_moments} are given by
\begin{align}
\rm{r-1}
  &= \frac{\E\bigl[\wt^{r-1}\bigr]}{\E\bigl[\wt\bigr]^{r-1}}
  = \frac{\E\bigl[\w^{r-1}\bigr]}{\E\bigl[\w\bigr]^{r-1}}
  = \prod_{r=0}^{r-2} \frac{\alpha + r}{\alpha + \beta + r} \,\Big/ \left(\frac{\alpha}{\alpha + \beta}\right)^{r-1}\\
  &= \frac{\Gamma(\alpha + r - 1)}{\Gamma(\alpha + \beta + r - 1)} \cdot \frac{\Gamma(\alpha + \beta)}{\Gamma(\alpha)} \cdot \left(\frac{\alpha}{\alpha + \beta}\right)^{r-1} \,.
\end{align}
Using Stirling's approximation, the above can be simplified for large $r$. This gives
\begin{equation}
\log(\rm{r-1})
  = - \beta \log(r) + \log\bigl(\Gamma(\alpha + \beta) / \Gamma(\alpha)\bigr) + (r-1)\log\left(\frac{\alpha}{\alpha + \beta}\right) + \smallO(1)\,.
\end{equation}
Therefore, the typical clique number $\tcn$ from Definition~\ref{def:typical_clique_number} is given by the solution in $r$ of
\begin{equation}
r = \frac{\log(n) - (\beta + 1)\log(r) + (r-1)\log\bigl(\frac{\alpha}{\alpha + \beta}\bigr) + \log\bigl(\Gamma(\alpha + \beta) / \Gamma(\alpha)\bigr) + 1}{\log\Bigl(\sf / \bigl(\frac{\alpha}{\alpha + \beta}\bigr)\Bigr)} + 1 + \smallO(1) \,.
\end{equation}
Solving this we obtain
\begin{equation}
\tcn = \frac{(1 + \beta)\,\lambertW_0\biggl(\frac{\bigl(n\,e\,\sf\,\Gamma(\alpha + \beta) / \Gamma(\alpha)\bigr)^{\frac{1}{1 + \beta}} \log(\sf)}{1 + \beta}\biggr)}{\log(\sf)} + \smallO(1)\,,
\end{equation}
As in the previous example, using the approximation $\lambertW_0(x) = \log(x) - \log\log(x) + \smallO(1)$, we obtain
\begin{align}
\tcn
  &= \frac{(1 + \beta) \log\Biggl(\frac{\bigl(n\,e\,\sf\,\Gamma(\alpha + \beta) / \Gamma(\alpha)\bigr)^{\frac{1}{1 + \beta}} \log(\sf)}{1 + \beta}\Biggr)}{\log(\sf)}\\
  &\qquad\qquad{}-\frac{(1 + \beta) \log\log\Biggl(\frac{\bigl(n\,e\,\sf\,\Gamma(\alpha + \beta) / \Gamma(\alpha)\bigr)^{\frac{1}{1 + \beta}} \log(\sf)}{1 + \beta}\Biggr)}{\log(\sf)} + \smallO(1)\\
  &= \frac{\log\bigl(n\,e\,\sf\,\Gamma(\alpha + \beta) / \Gamma(\alpha)\bigr) + (1 + \beta) \log\Bigl(\frac{\log(\sf)}{1 + \beta}\Bigr) - (1 + \beta) \log\log(n)}{\log(\sf)}  + \smallO(1)\\
  &= \log_{\sf}\bigl(n\,e\,\Gamma(\alpha + \beta) / \Gamma(\alpha)\bigr) - (1 + \beta) \log_{\sf}\bigl( (1 + \beta) \log_{\sf}(n)\bigr) + 1 + \smallO(1)\\[0.5ex]
  &= \log_{\sf}(n) - (1 + \beta) \log_{\sf}\bigl( (1 + \beta) \log_{\sf}(n)\bigr)\\
  &\qquad\qquad{} + \log_{\sf}(e) + \log_{\sf}\bigl(\Gamma(\alpha + \beta) / \Gamma(\alpha)\bigr) + 1 + \smallO(1) \,.
\end{align}

\subsection{Gamma weights}
\label{subsec:gamma_weights}
Let $\w$ have a Gamma distribution with shape $\alpha$ and rate $\beta$, that is $\w \sim \text{Gamma}(\alpha, \beta)$. First we assume that truncating the weight distribution has asymptotically almost no effect on the relative moments from Definition~\ref{def:relative_moments}. We begin by assuming that
\begin{equation}
\label{eq:assumption_rm_gamma}
\rm{r-1}
  = \frac{\E\bigl[\wt^{r-1}\bigr]}{\E\bigl[\wt\bigr]^{r-1}}
  = (1 + \smallO(1)) \, \frac{\E\bigl[\w^{r-1}\bigr]}{\E\bigl[\w\bigr]^{r-1}} \,,
\end{equation}
for all $r \leq \tcn$, and use this to find the typical clique number $\tcn$. After that, we will show that the assumption in \eqref{eq:assumption_rm_gamma} is valid. By the assumption in \eqref{eq:assumption_rm_gamma}
\begin{equation}
\rm{r-1}
  = (1 + \smallO(1)) \, \frac{\E\bigl[\w^{r-1}\bigr]}{\E\bigl[\w\bigr]^{r-1}}
  = (1 + \smallO(1)) \, \frac{\Gamma(\alpha+r-1)}{\Gamma(\alpha)\,\alpha^{r-1}} \,.
\end{equation}
To satisfy Assumption~\ref{ass:weights_probability_assumption} we must have $\sf \to \infty$, and therefore
\begin{equation}
\frac{\log(\rm{r-1})}{\log(\sf / \E[\wt])}
  = \frac{\log(\Gamma(\alpha+r-1)) - \log(\Gamma(\alpha)) - (r-1) \log(\alpha)}{\log(\sf / \E[\wt])} + \smallO(1) \,.
\end{equation}

Using Stirling's approximation again relying on the fact that $r$ is large, the typical clique number $\tcn$ is given by the solution in $r$ of
\begin{align}
r
  &= \frac{\log(n) - \log(r) + \log(\Gamma(\alpha+r-1)) - \log(\Gamma(\alpha)) - \left(r-1\right)\log(\alpha) + 1}{\log\bigl(\sf \, \beta / \alpha\bigr)} + 1 + \smallO(1)\\
  &= \frac{\log(n) - \log(r) + \bigl(\alpha+r-\frac{3}{2}\bigr)\log(\alpha+r-2) + 2}{\log\bigl(\sf \, \beta / \alpha\bigr)}\\
  &\qquad\qquad{} - \frac{(\alpha+r-2) + \log(\Gamma(\alpha)) + (r-1)\log(\alpha) + 2}{\log\bigl(\sf \, \beta / \alpha\bigr)} + 1 + \smallO(1)\\
  &= \frac{\log(n) + \bigl(\alpha+r-\frac{5}{2}\bigr)\log\bigl(\alpha+r-\frac{5}{2}\bigr)}{\log\bigl(\sf \, \beta / \alpha\bigr)}\\
  &\qquad\qquad{} - \frac{\bigl(\alpha+r-\frac{5}{2}\bigr) + \log(\Gamma(\alpha)) + (r-1)\log(\alpha)}{\log\bigl(\sf \, \beta / \alpha\bigr)}+ 1 + \smallO(1)\,.
\end{align}
Substituting $x = r + \alpha - 5/2$, we get
\begin{equation}
x = \frac{\log(n) + \bigl(\alpha - x - \frac{3}{2}\bigr)\log(\alpha) + x\log(x) - x - \log(\Gamma(\alpha)) + 2}{\log\bigl(\sf \, \beta / \alpha\bigr)} - \frac{3}{2} + \alpha + \smallO(1)\,.
\end{equation}
Solving for $x$ we find $\tcn + \alpha - 5/2$, and therefore the typical clique number $\tcn$ is given by
\begin{equation}
\label{eq:typical_clique_number_gamma_general}
\tcn = \frac{2\log(n)-(3-2\alpha)\log(\beta \, \sf)+4-2\log(\Gamma(\alpha))}{-\lambertW_{-1}\left(-\frac{2\log(n)-(3-2\alpha)\log(\beta \, \sf)+4-2\log(\Gamma(\alpha))}{2\,e\,\beta\,\sf}\right)} + \frac{5}{2} - \alpha + \smallO(1) \,,
\end{equation}
where $\lambertW_{-1}$ denotes the lower branch of the Lambert-W function, see \eqref{eq:lambert_w}.

\subsubsection{\texorpdfstring{First scaling: $\sf = (1+\phi)\log(n) / \beta$}{First scaling}}
\label{subsubsec:gamma_first_scaling}
Let $\sf = (1+\phi)\log(n) / \beta$, with $\phi > 0$. We will show that in this case \eqref{eq:typical_clique_number_gamma_general} simplifies to the result in Table~\ref{tbl:typical_clique_number_light_tailed1}. This gives
\begin{align}
\label{eq:typical_clique_number_gamma1}
\tcn
  &= \frac{\log(n) - \bigl(\frac{3}{2}-\alpha\bigr)\log((1+\phi)\log(n)) + 2 - \log(\Gamma(\alpha))}{-\lambertW_{-1}\left(-\frac{\log(n) - \bigl(\frac{3}{2}-\alpha\bigr)\log((1+\phi)\log(n)) + 2 - \log(\Gamma(\alpha))}{e \, (1+\phi)\log(n)}\right)} + \frac{5}{2} - \alpha + \smallO(1)\\
  &= \frac{\log(n) - \bigl(\frac{3}{2}-\alpha\bigr)\log((1+\phi)\log(n)) + 2 - \log(\Gamma(\alpha))}{-\lambertW_{-1}\left(-\frac{1}{e \, (1+\phi)}\right) + \smallO(1)} + \frac{5}{2} - \alpha + \smallO(1)\\
  &= (1 + \smallO(1))\,\frac{\log(n)}{-\lambertW_{-1}\left(-\frac{1}{e (1+\phi)}\right)}\,.
\end{align}

It remains to show that our assumption from \eqref{eq:assumption_rm_gamma} holds. We will do this in two parts: \emph{(i)} where we show $\smash{\E\bigl[\wt^{r-1}\bigr] / \E\bigl[\w^{r-1}\bigr]} \to 1$ for any $r \leq \tcn$; and \emph{(ii)} where we show $\smash{(\E[\wt] / \E[\w])^{r-1}} \to 1$ for any $r \leq \tcn$. First, observe that for any $k \geq 1$ we have
\begin{equation}
\label{eq:truncated_r1_moment_ratio_gamma}
\frac{\E\bigl[\wt^k\bigr]}{\E\bigl[\w^k\bigr]}
  = \frac{\E\bigl[\w^k \,\big|\, W \leq \frac{\sf}{1 + \delta}\bigr]}{\E\bigl[\w^k\bigr]}
  = \frac{1}{\P\bigl(W \leq \frac{\sf}{1 + \delta}\bigr)} \, \frac{\int_0^{\frac{\sf}{1 + \delta}} x^k f_{W}(x) dx}{\int_0^{\infty\;\;} x^k f_{W}(x) dx}
  = \frac{\P\bigl(Z_{k} \leq \frac{\sf}{1 + \delta}\bigr)}{\P\bigl(W \leq \frac{\sf}{1 + \delta}\bigr)} \,,
\end{equation}
where $Z_k \sim \text{Gamma}(\alpha + k, \beta)$.

\emph{Part (i): } 
To simplify notation, let $a \coloneqq (1+\phi)/(1+\delta)$, $b \coloneqq -1 / \lambertW_{-1}\left(-1 / (e (1+\phi))\right)$, and $z_n \coloneqq \omega_n + \alpha - 1 = (1 + \smallO(1)) b \log(n) + \alpha - 1$. Note that, because $\phi > \delta > 0$ we have
\begin{equation}
a = \frac{1+\phi}{1+\delta}
> 1
> \left(-\lambertW_{-1}\left(-\frac{1}{e (1+\phi)}\right)\right)^{-1}
= b \,.
\end{equation}
Finally, let $X_i \sim \text{Exp}(\beta)$. Then using \eqref{eq:truncated_r1_moment_ratio_gamma} and Assumption~\ref{ass:weights_probability_assumption} we have
\begin{align}
\frac{\E\bigl[\wt^{r-1}\bigr]}{\E\bigl[\w^{r-1}\bigr]}
  &= (1 + \smallO(1)) \, \P\Bigl(Z_{r-1} \leq \frac{\sf}{1 + \delta}\Bigr)\\
  &\geq (1 + \smallO(1)) \, \P\biggl(\sum_{i = 1}^{\lceil z_n \rceil} X_i \leq \frac{a}{\beta} \log(n) \biggr)\\
  &= (1 + \smallO(1)) \biggl(1 - \P\biggl(\frac{1}{\lceil z_n \rceil} \sum_{i = 1}^{\lceil z_n \rceil} X_i > (1 + \smallO(1)) \, \frac{1}{\beta} \, \frac{a}{b}\biggr)\biggr)\\
  \label{eq:weighted_gamma_chernoff_bound_gamma}
  &\geq (1 + \smallO(1)) \biggl(1 - \exp\biggl(- \lceil z_n \rceil \, I\Bigl((1 + \smallO(1)) \, \frac{1}{\beta} \, \frac{a}{b}\Bigr)\biggr)\biggr) \,,
\end{align}
where $I(x) \coloneqq x\beta - 1 - \log(x \beta)$ is the rate function of an exponential distribution with rate $\beta$. Hence, for $n$ large enough and because $a / b > 1$ we have
\begin{equation}
\label{eq:rate_bound_gamma}
I\Bigl((1 + \smallO(1)) \, \frac{1}{\beta} \, \frac{a}{b}\Bigr)
= (1 + \smallO(1)) \, I\Bigl(\frac{1}{\beta} \, \frac{a}{b}\Bigr)
= (1 + \smallO(1)) \, \Bigl((a / b) - 1 - \log(a / b)\Bigr)
> 0 \,.
\end{equation}
Combining \eqref{eq:weighted_gamma_chernoff_bound_gamma} and \eqref{eq:rate_bound_gamma} we see that $\smash{\E\bigl[\wt^{r-1}\bigr] / \E\bigl[\w^{r-1}\bigr]} \to 1$.

\emph{Part (ii): } 
From \eqref{eq:truncated_r1_moment_ratio_gamma} and integration by parts we obtain
\begin{align}
\frac{\E\bigl[\wt\bigr]}{\E\bigl[\w\bigr]}
  &= \frac{\P\bigl(Z_{1} \leq \frac{\sf}{1 + \delta}\bigr)}{\P\bigl(W \leq \frac{\sf}{1 + \delta}\bigr)}
  = \frac{\gamma(1 + \alpha, \beta \sf / (1 + \delta))}{\alpha \, \gamma(\alpha, \beta \sf / (1 + \delta))}\\
  &= 1 - \frac{(\beta \sf / (1 + \delta))^{\alpha} \exp(-\beta \sf / (1 + \delta))}{\alpha \, \gamma(\alpha, \beta \sf / (1 + \delta))}\\
  &= 1 - \bigO(1) \frac{\log(n)^{\alpha}}{n^a} \,,
\end{align}
where $\gamma(\cdot,\cdot)$ is the lower incomplete gamma function and we recall that $a = (1+\phi)/(1+\delta) > 1$. Hence, we have $\smash{(\E[\wt] / \E[\w])^{r-1}} \to 1$ for any $r \leq n$.

From parts \emph{(i)} and \emph{(ii)} we see that our assumption in \eqref{eq:assumption_rm_gamma} indeed holds.

\subsubsection{\texorpdfstring{Second scaling: $\sf = (1+\phi)\log(n) / \beta$}{Second scaling}}
\label{subsubsec:gamma_second_scaling}
Let $\sf = \log(n)^{1+\phi} / \beta$, with $\phi > 0$. We will show that in this case \eqref{eq:typical_clique_number_gamma_general} simplifies to the result in Table~\ref{tbl:typical_clique_number_light_tailed2}. This gives
\begin{align}
\label{eq:typical_clique_number_gamma2}
\tcn
  &= \frac{\log(n) - \bigl(\frac{3}{2}-\alpha\bigr)(1+\phi)\log\log(n) + 2 - \log(\Gamma(\alpha))}{-\lambertW_{-1}\left(-\frac{\log(n) - \bigl(\frac{3}{2}-\alpha\bigr)(1+\phi)\log\log(n) + 2 - \log(\Gamma(\alpha))}{e \, \log(n)^{1+\phi}}\right)} + \frac{5}{2} - \alpha + \smallO(1)\\
  &= \frac{\log(n) - \bigl(\frac{3}{2}-\alpha\bigr)(1+\phi)\log\log(n) + 2 - \log(\Gamma(\alpha))}{-\lambertW_{-1}\left(-\frac{1}{e \, \log(n)^{\phi}}\right) + \smallO(1)} + \frac{5}{2} - \alpha + \smallO(1)\\
  &= (1 + \smallO(1)) \, \frac{\log(n)}{\log\left(e \, \log(n)^{\phi}\right)}
  = (1 + \smallO(1)) \, \frac{1}{\phi} \, \frac{\log(n)}{\log\log(n)}\,.
\end{align}

Compared to Section~\ref{subsubsec:gamma_first_scaling} the scaling $\sf$ is larger and the typical clique number $\tcn$ is smaller. Therefore, it is evident that our assumption from \eqref{eq:assumption_rm_gamma} is also valid in this case.

\subsection{Half-normal weights}
\label{subsec:half-normal_weights}
Let $\w$ have a half-normal distribution with parameters $\mu = 0$ and $\sigma > 0$, that is $\w \sim |X|$, where $X \sim \textup{N}(0, \sigma)$. We proceed in the exact same way as for the Gamma distribution, and assume first that
\begin{equation}
\label{eq:assumption_rm_half_normal}
\rm{r-1}
  = \frac{\E\bigl[\wt^{r-1}\bigr]}{\E\bigl[\wt\bigr]^{r-1}}
  = (1 + \smallO(1)) \, \frac{\E\bigl[\w^{r-1}\bigr]}{\E\bigl[\w\bigr]^{r-1}} 
  = (1 + \smallO(1)) \, \pi^{\frac{r}{2}-1} \, \Gamma(r / 2) \,,
\end{equation}
for all $r \leq \tcn$. To satisfy Assumption~\ref{ass:weights_probability_assumption} we must have $\sf \to \infty$, and therefore
\begin{equation}
\frac{\log(\rm{r-1})}{\log(\sf / \E[\wt])}
= \frac{\left(\frac{r}{2}-1\right)\log(\pi) + \log(\Gamma(r / 2))}{\log(\sf / \E[\wt])} + \smallO(1) \,.
\end{equation}

Using Stirling's approximation and the fact that the typical clique number $\tcn$ grows with $n$, the typical clique number $\tcn$ is given by the solution in $r$ of
\begin{align}
r
  &= \frac{\log(n) - \log(r) + \left(\frac{r}{2}-1\right)\log(\pi) + \log\left(\Gamma\left(\frac{r}{2}\right)\right) + 1}{\log\bigl(\sf / \sqrt{2 \sigma^2/\pi}\bigr)} + 1 + \smallO(1)\\
  &= \frac{\log(n) - \log(r) + \left(\frac{r}{2}-1\right)\log(\pi) + \left(\frac{r}{2}-\frac{1}{2}\right)\log\left(\frac{r}{2}-1\right) - \frac{r}{2} + 3}{\log\bigl(\sf / \sqrt{2 \sigma^2/\pi}\bigr)} + 1 + \smallO(1)\\
  &= \frac{\log(n) + \left(\frac{r-3}{2}\right)\log\left(\frac{r-3}{2}\right) + \left(\frac{r-3}{2}\right)\bigl(\log(\pi) - 1\bigr) + \log\left(\frac{e^2\sqrt{\pi}}{2}\right)}{\log\bigl(\sf / \sqrt{2\sigma^2/\pi}\bigr)} + 1 + \smallO(1) \,.
\end{align}
Substituting $x = (r-3) / 2$, we get
\begin{equation}
x = \frac{1}{2}\,\frac{\log(n) + x\log(x) + x(\log(\pi) - 1) + \log\left(\frac{e^2\sqrt{\pi}}{2}\right)}{\log\bigl(\sf / \sqrt{2\sigma^2/\pi}\bigr)} - 1 + \smallO(1)\,.
\end{equation}
Solving for $x$ we find $(\tcn-3) / 2$, and therefore the typical clique number $\tcn$ is given by
\begin{equation}
\label{eq:typical_clique_number_half_normal_general}
\tcn =   \frac{2\log(n)-4\log(\sf)+4-\log(\pi)}{-\lambertW_{-1}\left(-\frac{2\log(n)-4\log(\sf)+4-\log(\pi)}{e \, \sf^2}\right)} + 3 + \smallO(1)\,,
\end{equation}
where $\lambertW_{-1}$ denotes the lower branch of the Lambert-W function, see \eqref{eq:lambert_w}.

\subsubsection{\texorpdfstring{First scaling: $\sf = (1+\phi)\sigma\sqrt{2 \log(n)}$}{First scaling}}
\label{subsubsec:half_normal_first_scaling}
Let $\sf = (1+\phi)\sigma\sqrt{2 \log(n)}$, with $\phi > 0$. We will show that in this case \eqref{eq:typical_clique_number_half_normal_general} simplifies to the result in Table~\ref{tbl:typical_clique_number_light_tailed1}. This gives
\begin{align}
\label{eq:typical_clique_number_half_normal1}
\tcn
  &= \frac{2\log(n)-4\log((1+\phi)\sqrt{2 \log(n)})+4-\log(\pi)}{-\lambertW_{-1}\left(-\frac{2\log(n)-4\log((1+\phi)\sqrt{2 \log(n)})+4-\log(\pi)}{2e \, (1+\phi)^2 \log(n)}\right)} + 3 + \smallO(1)\\
  &= \frac{2\log(n)-2\log\log(n)-2\log(\sqrt{4 \pi} (1+\phi)^2 / e^2)}{-\lambertW_{-1}\left(-\frac{1}{e (1+\phi)^2}\right) + \smallO(1)} + 3 + \smallO(1)\\
  &= (1 + \smallO(1))\,\frac{2\log(n)}{-\lambertW_{-1}\left(-\frac{1}{e (1+\phi)^2}\right)}\,.
\end{align}

\subsubsection{\texorpdfstring{Second scaling: $\sf = \sigma \sqrt{2 \log(n)}^{1+\phi}$}{Second scaling}}
\label{subsubsec:half_normal_second_scaling}
Let $\sf = \sigma \sqrt{2 \log(n)}^{1+\phi}$, with $\phi > 0$. We will show that in this case \eqref{eq:typical_clique_number_half_normal_general} simplifies to the result in Table~\ref{tbl:typical_clique_number_light_tailed2}. This gives
\begin{align}
\label{eq:typical_clique_number_half_normal2}
\tcn
  &= \frac{2\log(n)-2(1+\phi)\log(2\log(n))+4-\log(\pi)}{-\lambertW_{-1}\left(-\frac{2\log(n)-2(1+\phi)\log(2\log(n))+4-\log(\pi)}{e(2\log(n))^{1+\phi}}\right)} + 3 + \smallO(1)\\
  &= \frac{2\log(n)-2(1+\phi)\log(2\log(n))+4-\log(\pi)}{-\lambertW_{-1}\left(-\frac{1}{e(2\log(n))^{\phi}}\right) + \smallO(1)} + 3 + \smallO(1)\\
  &= (1 + \smallO(1)) \, \frac{2\log(n)}{\log\left(e(2\log(n))^{\phi}\right)}
  = (1 + \smallO(1)) \, \frac{2}{\phi} \, \frac{\log(n)}{\log\log(n)}\,.
\end{align}
At this stage, we have not yet shown that the assumption in \eqref{eq:assumption_rm_half_normal} holds. This can be done using a similar reasoning as for the Gamma distribution, and we omit the details
for brevity.

\subsection{Log-normal weights}
\label{subsec:log-normal_weights}
Let $\w$ have a log-normal distribution with parameters $\mu = 0$ and $\sigma = 1$, that is $\w \sim \exp(X)$, where $X$ is standard normal. Then one can show that the relative moments from Definition~\ref{def:relative_moments} are given by
\begin{equation}
\label{eq:assumption_rm_log_normal}
\frac{\log(\rm{r-1})}{\log(\sf / \E[\wt])}
  = \frac{\frac{1}{2} (r-1)(r-2)}{\log(\sf / \E[\wt])} + \smallO(1) \,.
\end{equation}
provided that $r \leq \tcn$. For brevity of presentation, we omit the details of this derivation.

Using this in Definition~\ref{def:typical_clique_number}, the typical clique number $\tcn$ is the solution in $r$ of
\begin{equation}
r = \frac{\log(n) - \log(r) + \frac{1}{2}(r-1)(r-2) + 1}{\log\bigl(\sf / \sqrt{e}\bigr)} + 1 + \smallO(1)\,.
\end{equation}
To solve this, we bound the solution with the following two bounds
\begin{align}
r &\leq \frac{\log(n) + \frac{1}{2}(r-1)(r-2) + 1}{\log\bigl(\sf / \sqrt{e}\bigr)} + 1 + \smallO(1)\,,\\
r &\geq \frac{\log(n) - r + \frac{1}{2}(r-1)(r-2) + 1}{\log\bigl(\sf / \sqrt{e}\bigr)} + 1 + \smallO(1)\,.
\end{align}
Solving the above gives
\begin{align}
\label{eq:typical_clique_number_upper_bound_log_normal}
\tcn &\leq \log(\sf) - \sqrt{\log(\sf)^2-2(\log(n)+1)} + 1 + \smallO(1)\,,\\
\label{eq:typical_clique_number_lower_bound_log_normal}
\tcn &\geq \log(\sf) - \sqrt{(1 + \log(\sf))^2-2\log(n)} + 2 + \smallO(1)\,.
\end{align}
Combining \eqref{eq:typical_clique_number_upper_bound_log_normal} and \eqref{eq:typical_clique_number_lower_bound_log_normal}, and plugging in $\sf = (1+\phi)\exp\bigl(\sqrt{2\log(n)}\bigr)$, we obtain the result in Table~\ref{tbl:typical_clique_number_light_tailed1}. Similarly, the result in Table~\ref{tbl:typical_clique_number_light_tailed2} is obtained by plugging in $\sf = \exp\bigl(\sqrt{2\log(n)}\bigr)^{1+\phi}$.

\paragraph{Acknowledgements.}
The work of RvdH is supported by NWO VICI grant 639.033.806 and NWO Gravitation Networks grant 024.002.003.

\printbibliography[title={References}]

\end{document}